\documentclass{amsart}
\usepackage{amsfonts}
\usepackage{amssymb}
\usepackage{amsmath}

\newtheorem{theorem}{Theorem}[section]
\theoremstyle{plain}

\newtheorem{definition}{Definition}[section]
\newtheorem{example}{Example}[section]

\newtheorem{remark}[theorem]{Remark}

\numberwithin{equation}{section}
\input{tcilatex}

\begin{document}
\title[Some basic measures in soft hybrid sets]{Similarity, entropy and
subsethood measures based on cardinality of soft hybrid sets}
\author{R\i dvan SAHIN}
\address{Department of Mathematics, Faculty of Science, Ataturk University,
Erzurum, 25240, Turkey.}
\email{mat.ridone@gmail.com }
\dedicatory{}

\begin{abstract}
The real world is inherently uncertain, imprecise and vague. Soft set theory
was firstly introduced by Molodtsov in 1999 as a general mathematical tool
for dealing with uncertainties, not clearly defined objects. A soft set
consists of two parts which are parameter set and approximate value set. So
while talking about any property on a soft set, it is notable to consider
that each parts should be evaluated separately. In this paper, by taking
into account this case, we firstly define the concept of cardinality of soft
hybrid sets which are soft set, fuzzy soft set, fuzzy parameterized soft set
and fuzzy parameterized fuzzy soft set. Then we discuss the entropy,
similarity and subsethood measures based on cardinality in a soft hybrid
set, and investigate the relationships among these concepts as well as
related examples. Finally, we present an application which is a
representation method based on cardinality of a soft hybrid space.
\end{abstract}

\maketitle

\section{Introduction}

The real world is full of uncertainty, imprecision and vagueness in fields
such as medical science, social science, engineering, economics etc.
Classical set theory, which is based on the crisp and exact case may not be
fully suitable for handling problems of uncertainty in such fields. So many
authors have become interested in modeling uncertainty recently and have
proposed various theories. Theory of fuzzy sets \cite{Zadeh}, theory of
intuitionistic fuzzy sets \cite{Atanassov1}, theory of vague sets \cite{Gau}
and theory of rough sets \cite{Pawlak} are some of the well-known theories.
In these theories, the concepts such as cardinality, entropy, distance
measure and similarity measure is widely used for the analysis and
representation of various types of data information such as numerical
information, interval-valued information, linguistic information, and so on.

The concept of cardinality expressing elementary characteristics of a set is
commonly used to characterize the concepts such as entropy, similarity,
subsethood and comparison between two fuzzy sets. The cardinality of a crisp
set is the number of elements in the set. In fuzzy set theory, since an
element can partially belong to a fuzzy set, a natural generalization of the
classical notion of cardinality is to weigh each element by its membership
degree. So cardinality of a fuzzy set is the sum of the membership values of
its all elements \cite{deluce}. In \cite{sostak}, Sostak studied on fuzzy
cardinals and cardinality of fuzzy sets.

Entropy and similarity measure of fuzzy sets are two basic concept in fuzzy
set theory. Entropy, which describes the degree of fuzziness in fuzzy set
and other extended higher order fuzzy sets was first mentioned by Zadeh \cite%
{Zadeh}. Then it have been widely investigated by many researchers from
different points of view. De Luca and Termini \cite{deluce} introduced some
axioms to describe the fuzziness degree of fuzzy set. Kaufmann \cite{kaufman}
proposed a method to measure the fuzziness degree of fuzzy set based on a
distance measure between its membership function and the membership function
of its nearest crisp set. In \cite{yager}, Yager introduced the fuzziness
degree of fuzzy set on the relationship between the fuzzy set and its
complement. On the other hand, a similarity measure is an important tool for
determining the degree of similarity between two objects and is applied in
many fields including pattern recognition, decision making, machine
learning, data mining, market prediction and image processing. Kosko \cite%
{kosko} presented a fuzzy entropy based on the concept of cardinality of the
fuzzy set. To show relationship between these two concepts, Liu \cite{liu}
investigated entropy, distance measure and similarity measure of fuzzy sets
and their relations. Fan and Xie \cite{fan} introduced the similarity
measure and fuzzy entropy induced by distance measure. Similarity measures
based on union and intersection operations, the maximum difference, and the
difference and sum of membership grades is proposed by Pappis and
Karacapilidis \cite{Papis}. Wang \cite{wang} presented two similarity
measures between fuzzy sets and between elements.

Subsethood is a concept used to measure the degree to which a set contains
another set. In classical theory, a set $A$ is called a subset of $B$ and is
denoted by $A\subset B$ if every element of $A$ is an element of $B$,
whenever $U$ is a universal set and $A,B$ are two sets in $U$. Therefore,
subsethood measure should be two valued for crisp sets. That is, either $A$
is precisely subset of $B$ or vice versa. But since an element $x$ in
universal set $U$ can belong to a fuzzy set $A$ to varying degrees, it is
notable to consider situations describing property being "more and less"
subset of a fuzzy set to another\ and to measure the degree of this
subsethood. Fuzzy subsethood allows a given fuzzy set to contain another to
some degree between $0$ and $1$.

Soft set theory \cite{D. Molodtsov} was firstly introduced by Molodtsov in
1999 as a general mathematical tool for dealing with uncertain, fuzzy, not
clearly defined objects. He showed several applications of this theory in
solving many practical problems in economics, engineering, social science,
medical science, etc. In recent years the development in the fields of soft
set theory is rapidly increasing and its application has been taking place
in a wide pace, see \cite{aygunoglu, kharal, D. Molodtsov, F. Feng, H.
Aktas, M. Ali, naim, naim2, naim3, naim4, maji, P. Maji2, P. Maji3,
Majumdar, sahin, wei, Y. B. Jun1}. However, in the literature, there is not
a lot of work on the concepts of cardinality, entropy, similarity and
subsethood measure of a soft set and its some hybrid structures which are
fuzzy soft sets, fuzzy parameterized soft set, fuzzy parameterized fuzzy
soft set, intuitionistic fuzzy soft set ect. Majumdar and Samanta \cite%
{Majumdar2} introduced the notion of softness of a soft set and used entropy
as a measure for this softness. In \cite{Majumdar}, Majumdar and Samanta
defined several types of distances between two soft sets and proposed
similarity measures of two soft sets. Also they showed an application of
this similarity measure of soft sets. Kharal \cite{kharal} presented that
some definition and lemma of Majumdar and Samanta contain errors.
Unfortunately, several basic properties presented in \cite{kharal} are not
true in general; these have been pointed out and improved by Yang \cite{wei}%
. Moreover, Yang \cite{wei} defined a new similarity measure of soft sets,
which measure similarity of both the parameter set and approximate value
set. If considering that a soft set consists of two parts, it is worth to
consider that each parts should be evaluated separately. To define the
concepts of cardinality, entropy, similarity and subsethood measure of a
soft hybrid set, we will adopt this idea in this paper.

The presentation of the rest of this paper is organized as follows. In
section 2, we briefly recall the notions of fuzzy set and soft hybrid sets
and their some properties. Also the concepts of cardinality, entropy,
similarity and subsethood measure of a fuzzy set are given in this section.
In section 3, we give definition of cardinality for soft hybrid sets and
investigate basic properties of the cardinality function. In section 4,
entropy of soft hybrid sets is presented and a theorem showing relative
between entropy and cardinality is provided. We investigate the similarity
and subsethood measures on soft hybrid sets in section 5,6 respectively.
Moreover, we provide that entropy can be expressed by the similarity
measure. In section 7, an application of cardinality is presented as a
method for representation of a soft hybrid spaces. Some concluding comments
are given in the final section,.

\section{Preliminaries}

In this section we present some basic concepts and terminology that will be
used throughout the paper.

In this paper, $U$ is a finite universe set, $P(U)$ its power set and $E$ is
always the finite universe set of parameters with respect to $U$ unless
otherwise specified.

\subsection{Fuzzy sets}

\begin{definition}
Zadeh \cite{Zadeh} defined a fuzzy set $A$ in the universe of discourse $U$
as follows:%
\begin{equation*}
A=\{x,%
\mu
_{A}(x):x\in U\}
\end{equation*}%
which is characterized by the membership function $%
\mu
_{A}(x):U\rightarrow \lbrack 0,1]$, where $%
\mu
_{A}(x)$ indicates the membership degree of the element $x$ to the set $A$.
The complement of a fuzzy set $A$ is defined by $A^{c}=\{x,1-%
\mu
_{A}(x):x\in U\}.$

A family of all fuzzy sets in $U$ will be denoted by $\mathcal{F}(U)$.
\end{definition}

\begin{definition}
Let $A$ and $B$ be two fuzzy set over $U$. Then

\begin{enumerate}
\item \cite{deluce} The sigma count of $A$, denoted by $\left\vert
A\right\vert $ is given by 
\begin{equation*}
\left\vert A\right\vert =\tsum count\left( A\right) =\tsum\nolimits_{x\in U}%
\mu
_{A}(x).
\end{equation*}

\begin{enumerate}
\item \cite{kosko2} If $A\subseteq B$, then $\left\vert A\right\vert \leq
\left\vert B\right\vert .$

\item \cite{dihar} If $A^{c}\subseteq B^{c}$, then $\left\vert
A^{c}\right\vert \leq \left\vert B^{c}\right\vert .$

\item \cite{kosko2} $\left\vert A\right\vert +\left\vert B\right\vert
=\left\vert A\cap B\right\vert +\left\vert A\cup B\right\vert .$

\item \cite{dihar} $\left\vert A^{c}\right\vert +\left\vert B^{c}\right\vert
=\left\vert A^{c}\cap B^{c}\right\vert +\left\vert A^{c}\cup
B^{c}\right\vert .$
\end{enumerate}

\item \cite{Zadeh} The entropy of $A$, denoted by $ent\left( A\right) $ is
given by%
\begin{equation*}
ent\left( A\right) =\frac{\left\vert A\cap A^{c}\right\vert }{\left\vert
A\cup A^{c}\right\vert }=\frac{\tsum\nolimits_{x\in U}\min \left( \mu
_{A}(x),1-\mu _{A}(x)\right) }{\tsum\nolimits_{x\in U}\max \left( \mu
_{A}(x),1-\mu _{A}(x)\right) }.
\end{equation*}

\item \cite{Papis} The similarity measure between $A$ and $B$, denoted by $%
sim\left( A,B\right) $ is given by%
\begin{equation*}
sim\left( A,B\right) =\frac{\left\vert A\cap B\right\vert }{\left\vert A\cup
B\right\vert }=\frac{\tsum\nolimits_{x\in U}\min \left( \mu _{A}(x),\mu
_{B}(x)\right) }{\tsum\nolimits_{x\in U}\max \left( \mu _{A}(x),\mu
_{B}(x)\right) }
\end{equation*}

\item \cite{kosko} The subsethood measure between $A$ and $B$, denoted by $%
sub\left( A,B\right) $ is given by%
\begin{equation*}
sub\left( A,B\right) =\frac{\left\vert A\cap B\right\vert }{\left\vert
A\right\vert }=\frac{\tsum\nolimits_{x\in U}\min \left( \mu _{A}(x),\mu
_{B}(x)\right) }{\tsum\nolimits_{x\in U}\mu _{A}(x)}
\end{equation*}
\end{enumerate}
\end{definition}

\subsection{\protect\bigskip Soft hybrid sets}

\begin{definition}
\cite{D. Molodtsov} A soft set $F_{A}$ over $U$ is a set defined by a
function $f_{A}$ representing a mapping 
\begin{equation*}
f_{A}:E\longrightarrow P(U)\text{ such that }f_{A}(e)=\varnothing \text{ if }%
e\notin A.
\end{equation*}%
Thus, a soft set $F_{A}$ over $U$ can be represented by the set of ordered
pairs%
\begin{equation*}
F_{A}=\{(e,f_{A}(e)):e\in E,\text{ }f(e)\in P(U)\}
\end{equation*}

Note that the set of all soft sets over U will be denoted by $\mathcal{S}%
(U). $
\end{definition}

\begin{example}
\label{soft ornek}Let $U=\left\{ x_{1},x_{2},x_{3},x_{4},x_{5}\right\} $ be
a universal set and $E=\left\{ e_{1},e_{2},e_{3},e_{4}\right\} $ be a set of
parameters. Consider $A=\{e_{1},e_{2},e_{4}\}$ and $B=%
\{e_{1},e_{2},e_{3},e_{4}\}$. Then

$\left( F_{A}\right) _{s}=\left\{ \left( e_{1},\left\{ x_{3},x_{4}\right\}
\right) ,\left( e_{2},\left\{ x_{1}\right\} \right) ,\left( e_{4},\left\{
x_{2},x_{4}\right\} \right) \right\} .$

$\left( G_{B}\right) _{s}=\left\{ \left( e_{1},\left\{
x_{3},x_{4},x_{5}\right\} \right) ,\left( e_{2},\left\{ x_{1},x_{3}\right\}
\right) ,\left( e_{3},\left\{ x_{1},x_{2},x_{4}\right\} \right) ,\left(
e_{4},U\right) \right\} .$
\end{example}

\begin{definition}
\cite{D. Molodtsov, naim} Let $F_{A}$ and $G_{B}$ be soft sets over a common
universe set $U$ and $A,B\subseteq E$. Then

\begin{enumerate}
\item $F_{A}$ is called the absolute soft set, denoted by $\tilde{U}$ if $%
f_{A}(e)=U$ for all $e\in E$.

$F_{\varnothing }$ is called the null soft set, denoted by $\Phi $ if $%
f_{\varnothing }(e)=\varnothing $ for all $e\in E$.

\item $F_{A}$ is a soft subset of $G_{B}$, denoted by $F_{A}\tilde{\subseteq}%
G_{B}$, if $f_{A}(e)\subseteq g_{B}(e)$ for all $e\in E$.

\item $F_{A}$ equals $G_{B}$, denoted by $F_{A}=G_{B}$, if $F_{A}\tilde{%
\subseteq}G_{B}$ and $G_{B}\tilde{\subseteq}F_{A}$.

\item The complement of $F_{A}$, denoted by $F_{A}^{c}$ is defined by $%
f_{A}^{c}(e)$ $=$ $U-f_{A}(e)$ for all $e\in E$.

\item The intersection of $F_{A}$ and $G_{B}$ is a soft set $K_{D}$ defined
by $k_{D}(e)=f_{A}(e)\cap g_{B}(e),$ where $D=A\cap B$. We write $%
(K_{D})=(F_{A})\tilde{\cap}(G_{B}).$

\item The union of $F_{A}$ and $G_{B}$ is a soft set $H_{C}$ defined by $%
h_{C}(e)=f_{A}(e)\cup g_{B}(e),$ where $C=A\cup B$. We write $(H_{C})=(F_{A})%
\tilde{\cup}(G_{B}).$

\item $F_{A}\tilde{\wedge}G_{B}$ is a soft set defined by $F_{A}\tilde{\wedge%
}G_{B}=(K,A\times B)$, where $k_{A\times B}(e,t)=f_{A}(e)\cap g_{B}(t)$ for
any $e\in A$ and $t\in B$.

\item $F_{A}\tilde{\vee}G_{B}$ is a soft set defined by $F_{A}\tilde{\vee}%
G_{B}=(H,A\times B)$, where $h_{A\times B}(e,t)=f_{A}(e)\cup g_{B}(t)$ for
any $e\in A$ and $t\in B$.
\end{enumerate}
\end{definition}

\begin{definition}
\cite{naim4} A fuzzy parameterized soft set $F_{A}$ over $U$ is a set
defined by a function $f_{A}$ representing a mapping%
\begin{equation*}
f_{A}:E\longrightarrow P(U)\text{ such that }f_{A}(e)=\varnothing \text{ if }%
\mu _{A}\left( e\right) =0
\end{equation*}%
where $A$ is a fuzzy set over $E$ with the membership function $\mu
_{A}:E\longrightarrow \lbrack 0,1]$.

Thus, a fuzzy parameterized soft set $F_{A}$ over $U$ can be represented by
the set of ordered pairs%
\begin{equation*}
F_{A}=\{(\mu _{A}\left( e\right) /e,\text{ }f_{A}(e)):e\in E,\text{ }%
f_{A}(e)\in P(U)\text{ and }\mu _{A}\left( e\right) \in \lbrack 0,1]\}
\end{equation*}

Note that the set of all fuzzy parameterized soft sets over $U$ will be
denoted by $\mathcal{FPS}(U)$.
\end{definition}

\begin{example}
\label{fuzzy parameterized soft ornek}Let $U=\left\{
x_{1},x_{2},x_{3},x_{4},x_{5}\right\} $ be a universal set and $E=\left\{
e_{1},e_{2},e_{3},e_{4}\right\} $ be a set of parameters. Consider \ $%
A=\{0.2/e_{2},0.6/e_{3},1/e_{4}\}$ and $B=\{0.3/e_{1},0.2/e_{2},0.6/e_{3}\}$%
. Then
\end{example}

$\left( F_{A}\right) _{fps}=\left\{ \left( 0.2/e_{2},\left\{
x_{2},x_{4}\right\} \right) ,\left( 0.6/e_{3},\left\{
x_{1},x_{3},x_{4}\right\} \right) \left( 1/e_{4},U\right) \right\} ,$

$\left( G_{B}\right) _{fps}=\{\left( 0.3/e_{1}\left\{ x_{1},x_{2}\right\}
\right) ,\left( 0.2/e_{2},\left\{ x_{4}\right\} \right) ,\left(
0.6/e_{3},\left\{ x_{1},x_{4}\right\} \right) \}$.

Now, we give some basic definitions on the fuzzy parameterized soft sets as
follows;

\begin{definition}
Let $F_{A}$ and $G_{B}$ be a fuzzy parameterized soft sets over $U$ and $%
A,B\subseteq E$. Then

\begin{enumerate}
\item $F_{A}$ is called the absolute fuzzy parameterized soft set, denoted
by $\tilde{U}$ if $f_{A}(e)=U$ \ and $\mu _{A}\left( e\right) =1$ for all $%
e\in E$.

$F_{\varnothing }$ is called the null fuzzy parameterized soft set, denoted
by $\Phi $ if $f_{\varnothing }(e)=\varnothing $ and $\mu _{A}\left(
e\right) =0$ for all $e\in E$.

\item $F_{A}$ is a fuzzy parameterized soft subset of $G_{B}$, denoted by $%
F_{A}\tilde{\subseteq}G_{B}$ if $\mu _{A}\left( e\right) \leq \mu _{B}\left(
e\right) $ and $f_{A}(e)\subseteq g_{B}(e)$ for all $e\in E$.

\item $F_{A}$ equals $G_{B}$, denoted by $F_{A}=G_{B}$, if $F_{A}\tilde{%
\subseteq}G_{B}$ and $G_{B}\tilde{\subseteq}F_{A}$.

\item The complement of $F_{A}$, denoted by $F_{A}^{c}$ is defined by $%
f_{A}^{c}(e)$ $=U-f_{A}(e)$ and $\mu _{A^{c}}\left( e\right) =$ $1-\mu
_{A}\left( e\right) $ for all $e\in E$.

\item The intersection of $F_{A}$ and $G_{B}$ is a fuzzy parameterized soft
set $K_{D}$ defined by $\mu _{D}(e)=\min \left\{ \mu _{A}\left( e\right)
,\mu _{B}\left( e\right) \right\} $ and $k_{D}(e)=f_{A}(e)\cap g_{B}(e)$ for
all $e\in E,$ where $D=A\cap B$. We write $K_{D}=F_{A}\tilde{\cap}G_{B}.$

\item The union of $F_{A}$ and $G_{B}$ is a fuzzy parameterized soft set $%
H_{C}$ defined by

$\mu _{C}(e)=\max \left\{ \mu _{A}\left( e\right) ,\mu _{B}\left( e\right)
\right\} $ and $h_{C}(e)=f_{A}(e)\cup g_{B}(e)$ for all $e\in E,$ where $%
C=A\cup B$. We write $H_{C}=F_{A}\tilde{\cup}G_{B}.$

\item $F_{A}\tilde{\wedge}G_{B}$ is a fuzzy parameterized soft set defined
by $F_{A}\tilde{\wedge}G_{B}=(K,A\times B)$, where $\mu _{A\times
B}(e,t)=\min \left\{ \mu _{A}\left( e\right) ,\mu _{B}\left( t\right)
\right\} $ and $k_{A\times B}(e,t)=f_{A}(e)\cap g_{B}(t)$ for any $e\in A$
and $t\in B$.

\item $F_{A}\tilde{\vee}G_{B}$ is a fuzzy parameterized soft set defined by $%
F_{A}\tilde{\vee}G_{B}=(H,A\times B)$, where $\mu _{A\times B}(e,t)=\max
\left\{ \mu _{A}\left( e\right) ,\mu _{B}\left( t\right) \right\} $ and $%
h_{A\times B}(e,t)=f_{A}(e)\cup g_{B}(t)$ for any $e\in A$ and $t\in B$.
\end{enumerate}
\end{definition}

\begin{definition}
\cite{P. Maji3} A fuzzy soft $F_{A}$ over $U$ is a set defined by a function 
$f_{A}$ representing a mapping%
\begin{equation*}
f_{A}:E\longrightarrow \mathcal{F}(U)\text{ sucht hat }f_{A}(e)=\varnothing 
\text{ if }e\notin A.
\end{equation*}%
where $f_{A}(e)$ is a fuzzy set over $U$ for all $e\in E.$

Thus, a fuzzy soft $F_{A}$ over $U$ can be represented by the set of ordered
pairs%
\begin{equation*}
F_{A}=\{(e,f_{A}(e)):e\in E,\text{ }f_{A}(e)\in \mathcal{F}(U)\}
\end{equation*}

Note that the set of all fuzzy soft sets over $U$ will be denoted by $%
\mathcal{FS}(U)$ .
\end{definition}

\begin{example}
\label{fuzzy soft orne} Let $U=\left\{ x_{1},x_{2},x_{3},x_{4},x_{5}\right\} 
$ be a universal set and $E=\left\{ e_{1},e_{2},e_{3},e_{4}\right\} $ be a
set of parameters. Consider $A=\{e_{2},e_{4}\}$ and $B=\{e_{1},e_{2},e_{4}%
\}. $ Then

$\left( F_{A}\right) _{fs}=\left\{ \left( e_{2},\left\{
0.1/x_{1},0.8/x_{3},0.3/x_{4}\right\} \right) ,\left( e_{4},\left\{
0.3/x_{1},0.4/x_{2}\right\} \right) \right\} $

$\left( G_{B}\right) _{fs}=\left\{ \left( e_{1},\left\{
0.3/x_{1},0.2/x_{2},0.7/x_{4}\right\} \right) ,\left( e_{2},\left\{
0.4/x_{1},0.5/x_{4}\right\} \right) ,\left( e_{4},\left\{
0.3/x_{2},0.2/x_{3},0.8/x_{4}\right\} \right) \right\} $
\end{example}

\begin{definition}
\cite{naim2, naim3} Let $F_{A}$ and $G_{B}$ be two fuzzy soft sets over $U$
and $A,B\subseteq E$. Then

\begin{enumerate}
\item $F_{A}$ is called the absolute fuzzy soft set, denoted by $\tilde{U}$
if $f_{\varnothing }(e)=U$ for all $e\in E,$ i.e., $f_{A}(e)(x)=1$ for all $%
e\in E$ and all $x\in U.$

$F_{\varnothing }$ is called the null fuzzy soft set, denoted by $\Phi $ if $%
f_{A}(e)=\varnothing $ for all $e\in E,$ i.e., $f_{\varnothing }(e)(x)=0$
for all $e\in E$ and all $x\in U.$

\item $F_{A}$ is a fuzzy soft subset of $G_{B}$, denoted by $F_{A}\tilde{%
\subseteq}G_{B}$, if $f_{A}(e)\subseteq g_{B}(e)$ for all $e\in E,$ i.e., $%
f_{A}(e)(x)\leq g_{B}(e)(x)$ for all $e\in E$ and all $x\in U.$

\item $F_{A}$ equals $G_{B}$, denoted by $F_{A}=G_{B}$, if $F_{A}\tilde{%
\subseteq}G_{B}$ and $G_{B}\tilde{\subseteq}F_{A}$.

\item The complement of $F_{A}$, denoted by $F_{A}^{c}$ is defined by $%
f_{A}^{c}(e)$ $=$ $U-f_{A}(e)$ for all $e\in E,$ i.e., $f_{A}^{c}(e)(x)$ $=$ 
$1-f_{A}(e)(x)$ for all $e\in E$ and all $x\in U$.

\item The intersection of $F_{A}$ and $G_{B}$ is a fuzzy soft set $K_{D}$
defined by $k_{D}(e)=f_{A}(e)\cap g_{B}(e)$ for all $e\in E,$ i.e., $%
k_{D}(e)(x)=\min \left\{ f_{A}(e)(x),g_{B}(e)(x)\right\} $ for all $e\in E$
and all $x\in U,$ where $D=A\cap B$. We write $K_{D}=F_{A}\tilde{\cap}G_{B}.$

\item The union of $F_{A}$ and $G_{B}$ is a fuzzy soft set $H_{C}$ defined
by $h_{C}(e)=f_{A}(e)\cup g_{B}(e)$ for all $e\in E$, i.e., $%
h_{C}(e)(x)=\max \left\{ f_{A}(e)(x),g_{B}(e)(x)\right\} $ for all $e\in E$
and all $x\in U,$ where $C=A\cup B$. We write $H_{C}=F_{A}\tilde{\cup}G_{B}.$

\item $F_{A}\tilde{\wedge}G_{B}$ is a fuzzy soft set defined by $F_{A}\tilde{%
\wedge}G_{B}=(K,A\times B)$, where $k_{A\times B}(e,t)=f_{A}(e)\cap g_{B}(t)$
for any $e\in A$ and $t\in B$, i.e., $k_{A\times B}(e)(x)=\min \left\{
f_{A}(e)(x),g_{B}(t)(x)\right\} $ for all $x\in U$.

\item $F_{A}\tilde{\vee}G_{B}$ is a fuzzy soft set defined by $F_{A}\tilde{%
\vee}G_{B}=(H,A\times B)$, where $h_{A\times B}(e,t)=f_{A}(e)\cup g_{B}(t)$
for any $e\in A$ and $t\in B$, i.e., $h_{A\times B}(e)(x)=\max \left\{
f_{A}(e)(x),g_{B}(t)(x)\right\} $ for all $x\in U$.
\end{enumerate}
\end{definition}

\begin{definition}
\cite{naim4} A fuzzy parameterized fuzzy soft set $F_{A}$ over $U$ is a set
defined by a function $f_{A}$ representing a mapping 
\begin{equation*}
f_{A}:E\longrightarrow \mathcal{F}(U)\text{ such that }f_{A}\left( e\right)
=\varnothing \text{ if }\mu _{A}\left( e\right) =0.
\end{equation*}%
where $A$ is a fuzzy set over $E$ with the membership function $\mu
_{A}:E\longrightarrow \lbrack 0,1]$ and $f_{A}\left( e\right) $ is a fuzzy
set over $U$ for all $e\in E$.

Thus, a fuzzy parameterized fuzzy soft set $F_{A}$ over $U$ can be
represented by the set of ordered pairs%
\begin{equation*}
F_{A}=\{(\mu _{A}\left( e\right) /e,\text{ }f_{A}(e)):e\in E,\text{ }%
f_{A}(e)\in \mathcal{F}(U)\text{ and }\mu _{A}\left( e\right) \in \lbrack
0,1]\}
\end{equation*}

Note that the set of all fuzzy parameterized fuzzy soft set over $U$ will be
denoted by $\mathcal{FPFS}(U)$.
\end{definition}

\begin{example}
\label{fuzzy parameterized fuzzy soft set ornek} Let $U=\left\{
x_{1},x_{2},x_{3},x_{4},x_{5}\right\} $ be a universal set and $E=\left\{
e_{1},e_{2},e_{3},e_{4}\right\} $ be a set of parameters. Consider $%
A=\{0.4/e_{1},0.2/e_{2}\}$ and $B=\{0.4/e_{1},0.2/e_{2},0.6/e_{3}\}.$ Then
\end{example}

$\left( F_{A}\right) _{fpfs}=\{\left( 0.4/e_{1},\left\{
0.3/x_{1},0.1/x_{2}\right\} \right) ,\left( 0.2/e_{2},\left\{
0.1/x_{2},0.4/x_{3},0.6/x_{4}\right\} \right) \}.$

$\left( G_{B}\right) _{fpfs}=\{\left( 0.4/e_{1},\left\{
0.2/x_{2},0.5/x_{3}\right\} \right) ,\left( 0.2/e_{2},\left\{
0.6/x_{3}\right\} \right) ,\left( 0.6/e_{3},\left\{ 0.2/x_{2}\right\}
\right) \}.$

\begin{definition}
\cite{naim4} Let $F_{A}$ and $G_{B}$ be fuzzy parameterized fuzzy soft set
over $U$ and $A,B\subseteq E$. Then

\begin{enumerate}
\item $F_{A}$ is called the absolute fuzzy parameterized fuzzy soft set,
denoted by $\tilde{U}$ if $f_{A}(e)=U$ \ and $\mu _{A}\left( e\right) =1$
for all $e\in E$ i.e., $f_{A}(e)(x)=1$ for all $e\in E$ and all $x\in U.$

$F_{\varnothing }$ is called the null fuzzy parameterized fuzzy soft set,
denoted by $\Phi $ if $f_{\varnothing }(e)=\varnothing $ and $\mu _{A}\left(
e\right) =0$ for all $e\in E$, i.e., $f_{\varnothing }(e)(x)=0$ for all $%
e\in E$ and all $x\in U.$

\item $F_{A}$ is a fuzzy parameterized fuzzy soft subset of $G_{B}$, denoted
by $F_{A}\tilde{\subseteq}G_{B}$, if $\mu _{A}\left( e\right) \leq \mu
_{B}\left( e\right) $ and $f_{A}(e)\subseteq g_{B}(e)$ for all $e\in E$,
i.e., $f_{A}(e)(x)\leq g_{B}(e)(x)$ for all $e\in E$ and all $x\in U.$

\item $F_{A}$ equals $G_{B}$, denoted by $F_{A}=G_{B}$, if $F_{A}\tilde{%
\subseteq}G_{B}$ and $G_{B}\tilde{\subseteq}F_{A}$.

\item The complement of $F_{A}$, denoted by $F_{A}^{c}$ is defined by $%
f_{A}^{c}(e)$ $=U-f_{A}(e)$ and $\mu _{A^{c}}\left( e\right) =$ $1-\mu
_{A}\left( e\right) $ for all $e\in E$, i.e., $f_{A}^{c}(e)(x)$ $%
=1-f_{A}(e)(x)$ for all $e\in E$ and all $x\in U.$

\item The intersection of $F_{A}$ and $G_{B}$ is a fuzzy parameterized fuzzy
soft set $K_{D}$ defined by $\mu _{D}(e)=\min \left\{ \mu _{A}\left(
e\right) ,\mu _{B}\left( e\right) \right\} $ and $k_{D}(e)=f_{A}(e)\cup
g_{B}(e)$ for all $e\in E$, i.e., $k_{D}(e)(x)=\min \left\{
f_{A}(e)(x),g_{B}(e)(x)\right\} $ for all $e\in E$ and $x\in U,$ where $%
D=A\cap B$. We write $K_{D}=F_{A}\tilde{\cap}G_{B}.$

\item The union of $F_{A}$ and $G_{B}$ is a fuzzy parameterized fuzzy soft
set $H_{C}$ defined by $\mu _{C}(e)=\max \left\{ \mu _{A}\left( e\right)
,\mu _{B}\left( e\right) \right\} $ and $h_{C}(e)=f_{A}(e)\cup g_{B}(e)$ for
all $e\in E$, i.e.,

$h_{C}(e)(x)=\max \{f_{A}(e)(x),g_{B}(e)(x)\}$ for all $e\in E$ and $x\in U,$
where $C=A\cup B$. We write $H_{C}=F_{A}\tilde{\cup}G_{B}.$

\item $F_{A}\tilde{\wedge}G_{B}$ is a fuzzy parameterized fuzzy soft set
defined by $F_{A}\tilde{\wedge}G_{B}=(K,A\times B)$, where $\mu _{A\times
B}(e,t)=\min \left\{ \mu _{A}\left( e\right) ,\mu _{B}\left( t\right)
\right\} $ and $k_{A\times B}(e,t)=f_{A}(e)\cap g_{B}(t)$ for any $e\in A$
and $t\in B$, i.e., $k_{A\times B}(e)(x)=\min \left\{
f_{A}(e)(x),g_{B}(e)(x)\right\} $ for all $x\in U$ (where $\cap $ is the
intersection operation of sets).

\item $F_{A}\tilde{\vee}G_{B}$ is a fuzzy parameterized fuzzy soft set
defined by $F_{A}\tilde{\vee}G_{B}=(H,A\times B)$, where $\mu _{A\times
B}(e,t)=\max \left\{ \mu _{A}\left( e\right) ,\mu _{B}\left( t\right)
\right\} $ and $h_{A\times B}(e,t)=f_{A}(e)\cup g_{B}(t)$ for any $e\in A$
and $t\in B$, i.e., $h_{A\times B}(e)(x)=\max \left\{
f_{A}(e)(x),g_{B}(e)(x)\right\} $ for all $x\in U$ (where $\cup $ is the
intersection operation of sets).
\end{enumerate}
\end{definition}

Here, we mean the concepts of soft set$,$ fuzzy soft set$,$ fuzzy
parameterized soft set and fuzzy parameterized fuzzy soft set by soft hybrid
sets. By $X(U),$ We consider any from $\mathcal{S}(U),$ $\mathcal{FS}(U),$ $%
\mathcal{FPS}(U)$ and $\mathcal{FPFS}(U).$ According to above-given
definitions in same universal set, we have the ranking soft set $%
\Longrightarrow $ fuzzy soft set $\Longrightarrow $ fuzzy parameterized soft
set $\Longrightarrow $ fuzzy parameterized fuzzy soft set. So we will
satisfy the proofs over most general set, fuzzy parameterized fuzzy soft set.

\section{\protect\bigskip Cardinality of soft hybrid sets}

The cardinality of a set in the crisp sense plays an important role in
Mathematics and is defined as the number of elements in the set. In fuzzy
set theory, the concept of cardinality of a fuzzy set is an extension of the
count of elements of a crisp set. A simple way of extending the concept of
cardinality was suggested by Deluca and Termini \cite{deluce}. In the
section, we generalize the concept of cardinality to soft hybrid sets.

\begin{definition}
Let $\left( a_{1},b_{1}\right) ,\left( a_{2},b_{2}\right) \in 
\mathbb{R}
\times 
\mathbb{R}
.$ Then we define

\begin{enumerate}
\item $\left( a_{1},b_{1}\right) \leq \left( a_{2},b_{2}\right) $ iff $%
a_{1}\leq a_{2},b_{1}\leq b_{2}$ and $\left( a_{1},b_{1}\right) =\left(
a_{2},b_{2}\right) $ iff $a_{1}=a_{2},b_{1}=b_{2}.$

\item $\left( a_{1},b_{1}\right) +\left( a_{2},b_{2}\right) =\left(
a_{1}+a_{2},b_{1}+b_{2}\right) .$
\end{enumerate}
\end{definition}

\begin{definition}
Let $F_{A}$ be a soft hybrid set over $U$ , i.e., $F_{A}\in X(U)$. Let $%
count(\pi ,\sigma )$ be a mapping given by $count(\pi ,\sigma ):$ $%
X(U)\longrightarrow \left( 
\mathbb{R}
^{+}\cup \left\{ 0\right\} \right) \times \left( 
\mathbb{R}
^{+}\cup \left\{ 0\right\} \right) $, where $\pi :P(E)$ $\longrightarrow 
\mathbb{R}
^{+}\cup \left\{ 0\right\} $ $($or $\pi :\mathcal{F}(E)$ $\longrightarrow 
\mathbb{R}
^{+}\cup \left\{ 0\right\} )$ and $\sigma :P(U)\longrightarrow 
\mathbb{R}
^{+}\cup \left\{ 0\right\} $ $($or $\sigma :\mathcal{F}(U)\longrightarrow 
\mathbb{R}
^{+}\cup \left\{ 0\right\} )$ are two mappings. Then $\left\vert
F_{A}\right\vert $ is called the cardinality of $F_{A}$ and is defined by%
\begin{equation*}
\left\vert F_{A}\right\vert =count(\pi ,\sigma )(F_{A})=\left( \pi
(A),\sigma (F)\right) =\left( \left\vert A\right\vert ,\left\vert
F\right\vert \right)
\end{equation*}%
where $\left\vert A\right\vert $ is the cardinality of $A$ and $\left\vert
F\right\vert $ is the sum of cardinalities of $f_{A}(e),\forall e\in A.$
\end{definition}

\begin{definition}
Let $F_{A}$ be a soft hybrid set over $U$, i.e., $F_{A}\in X(U)$.

\begin{enumerate}
\item The function $\left\vert F_{A}\right\vert $ defined by $\left\vert
F_{A}\right\vert =count(\pi ,\sigma )(F_{A})=\left( \pi (A),\sigma
(F)\right) =\left( \left\vert A\right\vert ,\left\vert F\right\vert \right)
=\left( \left\vert A\right\vert ,\tsum\nolimits_{e\in A}\left\vert
f_{A}(e)\right\vert \right) $ where $f_{A}(e)\in P(U)$ for all $e\in E,$ is
called the cardinality of soft set $F_{A}.$

\item The function $\left\vert F_{A}\right\vert $ defined by $\left\vert
F_{A}\right\vert =count(\pi ,\sigma )(F_{A})=\left( \pi (A),\sigma
(F)\right) =\left( \left\vert A\right\vert ,\left\vert F\right\vert \right)
=\left( \tsum\nolimits_{e\in A}\left\vert \mu _{A}(e)\right\vert
,\tsum\nolimits_{e\in A}\left\vert f_{A}(e)\right\vert \right) $ where $%
f_{A}(e)\in P(U)$ and $\mu _{A}(e)\in \mathcal{F}(E)$ for all $e\in E$, is
called the cardinality of fuzzy parameterized soft set $F_{A}.$

\item The function $\left\vert F_{A}\right\vert $ defined by $\left\vert
F_{A}\right\vert =count(\pi ,\sigma )(F_{A})=\left( \pi (A),\sigma
(F)\right) =\left( \left\vert A\right\vert ,\left\vert F\right\vert \right)
=\left( \left\vert A\right\vert ,\tsum\nolimits_{e\in A}\tsum\nolimits_{x\in
U}\left\vert f_{A}(e)(x)\right\vert \right) $ where $f_{A}(e)\in \mathcal{F}%
(U)$ for all $e\in E,$ is called the cardinality of fuzzy soft set $F_{A}.$

\item The function $\left\vert F_{A}\right\vert $ defined by $\left\vert
F_{A}\right\vert =count(\pi ,\sigma )(F_{A})=\left( \pi (A),\sigma
(F)\right) =\left( \left\vert A\right\vert ,\left\vert F\right\vert \right)
=\left( \tsum\nolimits_{e\in A}\left\vert \mu _{A}(e)\right\vert
,\tsum\nolimits_{e\in A}\tsum\nolimits_{x\in U}\left\vert
f_{A}(e)(x)\right\vert \right) $ where $f_{A}(e)\in \mathcal{F}(U)$ and $\mu
_{A}(e)\in \mathcal{F}(E)$ for all $e\in E$, is called the cardinality of
fuzzy parameterized fuzzy soft set $F_{A}.$
\end{enumerate}
\end{definition}

\begin{example}
Consider the soft hybrid sets $(F_{A})_{s}$,$(F_{A})_{fps}$, $(F_{A})_{fs}$
and $(F_{A})_{fpfs}$. Then

$%
\begin{array}{cccc}
\left\vert (F_{A})_{s}\right\vert =\left( 3,5\right) & \left\vert
(F_{A})_{fs}\right\vert =\left( 2,1.9\right) & \left\vert
(F_{A})_{fps}\right\vert =\left( 1.8,10\right) & \left\vert
(F_{A})_{fpfs}\right\vert =\left( 0.6,1.5\right) .%
\end{array}%
$
\end{example}

\begin{theorem}
Let $F_{A}$ and $G_{B}$ be two soft hybrid sets over $U$ , i.e., $%
F_{A},G_{B}\in X(U)$.

\begin{enumerate}
\item If $F_{A}\tilde{\subseteq}G_{B}$ then $\left\vert F_{A}\right\vert
\leq \left\vert G_{B}\right\vert $ and if $F_{A}^{c}\tilde{\subseteq}%
G_{B}^{c}$ then $\left\vert F_{A}^{c}\right\vert \leq \left\vert
G_{B}^{c}\right\vert $.

\item If $F_{A}\tilde{\subseteq}G_{B}$ and $G_{B}=\tilde{U}$, then $%
\left\vert F_{A}\right\vert \leq \left\vert \tilde{U}\right\vert .$

\item If $F_{A}=\Phi $ iff $\left\vert F_{A}\right\vert =\left( 0,0\right)
=0.$

\item If $\left\vert F_{A}\right\vert =\left( a,b\right) $ then $\left(
a,b\right) \leq \left( m,mn\right) $, where $\left\vert E\right\vert =m$ and 
$\left\vert U\right\vert =n.$
\end{enumerate}
\end{theorem}

\begin{remark}
Let $F_{A}$ and $G_{B}$ be two soft hybrid sets over $U$ , i.e., $%
F_{A},G_{B}\in X(U)$. Then $F_{A}$ and $G_{B}$ need not to be equal even if
they have same cardinality.
\end{remark}

\begin{definition}
Let $F_{A}$ and $G_{B}$ be two soft hybrid sets over $U$ , i.e., $%
F_{A},G_{B}\in X(U).$ We define the soft hybrid sets $F_{A\times B}$ and $%
G_{A\times B}$ as follows:

$F_{A\times B}=\left\{ (\left( e,t\right) ,\text{ }f_{A}(e)):\left(
e,t\right) \in A\times B\right\} ,$

$G_{A\times B}=\left\{ (\left( e,t\right) ,\text{ }g_{B}(t)):\left(
e,t\right) \in A\times B\right\} .$

In fact, this is a simple operation which is the parameter reduction.
\end{definition}

\begin{theorem}
\label{teo}Let $F_{A}$ and $G_{B}$ be two soft hybrid sets over $U$ , i.e., $%
F_{A},G_{B}\in X(U)$.

\begin{theorem}
$%
\begin{array}[t]{ll}
(1)\left\vert F_{A}\tilde{\cup}G_{B}\right\vert +\left\vert F_{A}\tilde{\cap}%
G_{B}\right\vert =\left\vert F_{A}\right\vert +\left\vert G_{B}\right\vert & 
(2)\left\vert F_{A}^{c}\tilde{\cup}G_{B}^{c}\right\vert +\left\vert F_{A}^{c}%
\tilde{\cap}G_{B}^{c}\right\vert =\left\vert F_{A}^{c}\right\vert
+\left\vert G_{B}^{c}\right\vert \\ 
(3)\left\vert F_{A}\tilde{\vee}G_{B}\right\vert +\left\vert F_{A}\tilde{%
\wedge}G_{B}\right\vert =\left\vert F_{A\times B}\right\vert +\left\vert
G_{A\times B}\right\vert & (4)\left\vert F_{A}^{c}\tilde{\vee}%
G_{B}^{c}\right\vert +\left\vert F_{A}^{c}\tilde{\wedge}G_{B}^{c}\right\vert
=\left\vert F_{A\times B}^{c}\right\vert +\left\vert G_{A\times
B}^{c}\right\vert%
\end{array}%
$
\end{theorem}
\end{theorem}

\begin{proof}
Let $F_{A},G_{B}\in X(U).$ For all $e\in E$ and all $x\in U,$
\end{proof}

\ $(1)$ $\left\vert F_{A}\tilde{\cup}G_{B}\right\vert +\left\vert F_{A}%
\tilde{\cap}G_{B}\right\vert =count(\pi ,\sigma )(F_{A}\tilde{\cup}%
G_{B})+count(\pi ,\sigma )(F_{A}\tilde{\cap}G_{B})$

$=\left( \left\vert A\cup B\right\vert ,\left\vert F\cup G\right\vert
\right) +\left( \left\vert A\cap B\right\vert ,\left\vert F\cap G\right\vert
\right) =\left( \left\vert A\cup B\right\vert +\left\vert A\cap B\right\vert
,\left\vert F\cup G\right\vert +\left\vert F\cap G\right\vert \right) $

$=(\left( \left\vert \max \left\{ \mu _{A}(e),\mu _{B}(e)\right\}
\right\vert +\left\vert \min \left\{ \mu _{A}(e),\mu _{B}(e)\right\}
\right\vert \right) ,(\left\vert \max \left\{
f_{A}(e)(x),g_{B}(e)(x)\right\} \right\vert $

$+\left\vert \min \left\{ f_{A}(e)(x),g_{B}(e)(x)\right\} \right\vert
)=\left( \left\vert \mu _{B}(e)\right\vert +\left\vert \mu
_{A}(e)\right\vert \right) ,\left( \left\vert g_{B}(e)(x)\right\vert
+\left\vert f_{A}(e)(x)\right\vert \right) $

$=\left( \left\vert \mu _{A}(e)\right\vert ,\left\vert
f_{A}(e)(x)\right\vert \right) +\left( \left\vert \mu _{B}(e)\right\vert
,\left\vert g_{B}(e)(x)\right\vert \right) =count(\pi ,\sigma
)(F_{A})+count(\pi ,\sigma )(G_{B})=\left\vert F_{A}\right\vert +\left\vert
G_{B}\right\vert .$

$(2)$ Similarly, we can easily get that $\left\vert F_{A}^{c}\tilde{\cup}%
G_{B}^{c}\right\vert +\left\vert F_{A}^{c}\tilde{\cap}G_{B}^{c}\right\vert
=\left\vert F_{A}^{c}\right\vert +\left\vert G_{B}^{c}\right\vert .$

\ $(3)$ For all $\left( e,t\right) \in A\times B$ and all $x\in U$,

$\left\vert F_{A}\tilde{\vee}G_{B}\right\vert +\left\vert F_{A}\tilde{\wedge}%
G_{B}\right\vert =count(\pi ,\sigma )(F_{A}\tilde{\vee}G_{B})+count(\pi
,\sigma )(F_{A}\tilde{\wedge}G_{B})$

$=\left( \left\vert A\times B\right\vert ,\left\vert F\cup G\right\vert
\right) +\left( \left\vert A\times B\right\vert ,\left\vert F\cap
G\right\vert \right) =\left( \left( \left\vert A\times B\right\vert
+\left\vert A\times B\right\vert \right) ,\left( \left\vert F\cup
G\right\vert +\left\vert F\cap G\right\vert \right) \right) $

$=(\left( \left\vert A\times B\right\vert +\left\vert A\times B\right\vert
\right) ,(\left\vert \max \left\{ f_{A}(e)(x),g_{B}(t)(x)\right\}
\right\vert +\left\vert \min \left\{ f_{A}(e)(x),g_{B}(t)(x)\right\}
\right\vert )$

$=\left( \left( \left\vert A\times B\right\vert +\left\vert A\times
B\right\vert \right) ,\left( \left\vert g_{B}(t)(x)\right\vert +\left\vert
f_{A}(e)(x)\right\vert \right) \right) $

$=\left( \left\vert A\times B\right\vert ,\left\vert f_{A}(e)(x)\right\vert
\right) +\left( \left\vert A\times B\right\vert ,\left\vert
g_{B}(t)(x)\right\vert \right) $

$=count(\pi ,\sigma )(F_{A\times B})+count(\pi ,\sigma )(G_{A\times
B})=\left\vert F_{A\times B}\right\vert +\left\vert G_{A\times B}\right\vert
.$

$(4)$ Similarly, we can easily get that $\left\vert F_{A}^{c}\tilde{\vee}%
G_{B}^{c}\right\vert +\left\vert F_{A}^{c}\tilde{\wedge}G_{B}^{c}\right\vert
=\left\vert F_{A\times B}^{c}\right\vert +\left\vert G_{A\times
B}^{c}\right\vert .$

\section{Entropy of soft hybrid sets}

Now, we give a new definition to measure the softness of a soft hybrid set.

\begin{definition}
\label{entropy def} Let $F_{A}$ be a soft hybrid set over $U$ , i.e., $%
F_{A}\in X(U)$. Let $ent(\varepsilon ,\kappa )$ be a mapping given by $%
ent(\varepsilon ,\kappa ):X(U)\longrightarrow \left[ 0,1\right] \times \left[
0,1\right] $ where $\varepsilon :P(E)$ $\longrightarrow \left[ 0,1\right] $ $%
($or $\varepsilon :\mathcal{F}(E)$ $\longrightarrow \left[ 0,1\right] )$ and 
$\kappa :P(U)\longrightarrow \left[ 0,1\right] $ $($ or $\kappa :\mathcal{F}%
(U)\longrightarrow \left[ 0,1\right] )$ are two mappings. It is called an
entropy for $F_{A}$, if it satisfies the following axiomatic reguirements:

\begin{enumerate}
\item $ent(\varepsilon ,\kappa )(F_{A})=\left( 0,0\right) =0$ if and only if 
$F_{A}$ is a soft set.

\item $ent(\varepsilon ,\kappa )(F_{A})=\left( 1,1\right) =1$ if and only if 
$A=A^{c}$ and $F=F^{c}$.

\item $ent(\varepsilon ,\kappa )(F_{A})=ent(\varepsilon ,\kappa )F_{A}^{c}.$

\item $ent(\varepsilon ,\kappa )(F_{A})\leq ent(\varepsilon ,\kappa )(G_{B})$
if $\mu _{A}(e)\leq \mu _{B}(e)\leq 0.5$ and $f_{A}(e)(x)\leq
g_{B}(e)(x)\leq 0.5$ or $0.5\leq \mu _{B}(e)\leq \mu _{A}(e)$ and $0.5\leq
g_{B}(e)(x)\leq f_{A}(e)(x)$ for all $e\in E$ and all $x\in U.$
\end{enumerate}
\end{definition}

\begin{definition}
Let $F_{A}$ be a soft hybrid set over $U$ , i.e., $F_{A}\in X(U)$. Then $%
ent(\varepsilon ,\kappa )$ is defined as follows:%
\begin{equation*}
ent(\varepsilon ,\kappa )(F_{A})=\left( \varepsilon (A),\kappa (F)\right)
=\left( \frac{\pi \left( A\cap A^{c}\right) }{\pi \left( A\cup A^{c}\right) }%
,\frac{\sigma \left( F\cap F^{c}\right) }{\sigma \left( F\cup F^{c}\right) }%
\right)
\end{equation*}%
where mapping $count(\pi ,\sigma )$ is a cardinal function.
\end{definition}

\begin{theorem}
The above-defined measure $ent(\varepsilon ,\kappa )(F_{A})$ is an entropy
of soft hybrid set $F_{A},$ i.e., is satisfies all the properties in
Definition \ref{entropy def}.
\end{theorem}

$(1)$ $ent(\varepsilon ,\kappa )(F_{A})=\left( 0,0\right)
\Longleftrightarrow \left( \frac{\pi \left( A\cap A^{c}\right) }{\pi \left(
A\cup A^{c}\right) },\frac{\sigma \left( F\cap F^{c}\right) }{\sigma \left(
F\cup F^{c}\right) }\right) \Longleftrightarrow \frac{\pi \left( A\cap
A^{c}\right) }{\pi \left( A\cup A^{c}\right) }=0$ and $\frac{\sigma \left(
F\cap F^{c}\right) }{\sigma \left( F\cup F^{c}\right) }=0\Longleftrightarrow
\pi \left( A\cap A^{c}\right) =0$ and $\sigma \left( F\cap F^{c}\right)
=0\Longleftrightarrow \left\vert A\cap A^{c}\right\vert =0$ and $\left\vert
F\cap F^{c}\right\vert =0\Longleftrightarrow \left\vert \min \left\{ \mu
_{A}(e),1-\mu _{A}(e)\right\} \right\vert =0$ and $\left\vert \min \left\{
f_{A}(e)(x),1-f_{A}(e)(x)\right\} \right\vert =0$ for all $e\in E$ and all $%
x\in U\Longleftrightarrow \mu _{A}(e)=0$ or $\mu _{A}(e)=1$ and $%
f_{A}(e)(x)=0$ or $f_{A}(e)(x)=1$ for all $e\in E$ and all $x\in
U\Longleftrightarrow F_{A}$ is a soft set.

$(2)$ $ent(\varepsilon ,\kappa )(F_{A})=\left( 1,1\right)
\Longleftrightarrow \left( \frac{\pi \left( A\cap A^{c}\right) }{\pi \left(
A\cup A^{c}\right) },\frac{\sigma \left( F\cap F^{c}\right) }{\sigma \left(
F\cup F^{c}\right) }\right) \Longleftrightarrow \frac{\pi \left( A\cap
A^{c}\right) }{\pi \left( A\cup A^{c}\right) }=1$ and $\frac{\sigma \left(
F\cap F^{c}\right) }{\sigma \left( F\cup F^{c}\right) }=1\Longleftrightarrow
\pi \left( A\cap A^{c}\right) =\pi \left( A\cup A^{c}\right) $ and $\sigma
\left( F\cap F^{c}\right) =\sigma \left( F\cup F^{c}\right)
\Longleftrightarrow \left\vert A\cap A^{c}\right\vert =\left\vert A\cup
A^{c}\right\vert $ and $\left\vert F\cap F^{c}\right\vert =\left\vert F\cup
F^{c}\right\vert \Longleftrightarrow \left\vert \min \left\{ \mu
_{A}(e),1-\mu _{A}(e)\right\} \right\vert =\left\vert \max \left\{ \mu
_{A}(e),1-\mu _{A}(e)\right\} \right\vert $ and $\left\vert \min \left\{
f_{A}(e)(x),1-f_{A}(e)(x)\right\} \right\vert =\left\vert \max \left\{
f_{A}(e)(x),1-f_{A}(e)(x)\right\} \right\vert $ for all $e\in E$ and all $%
x\in U\Longleftrightarrow A=A^{c}$ and $F=F^{c}.$

\bigskip $(3)$ For all $e\in E$ and all $x\in U$,

$ent(\varepsilon ,\kappa )(F_{A})=\left( \varepsilon (A),\kappa (F)\right)
=\left( \frac{\pi \left( A\cap A^{c}\right) }{\pi \left( A\cup A^{c}\right) }%
,\frac{\sigma \left( F\cap F^{c}\right) }{\sigma \left( F\cup F^{c}\right) }%
\right) $

$=\left( \frac{\left\vert \min \left\{ \mu _{A}(e),1-\mu _{A}(e)\right\}
\right\vert }{\left\vert \max \left\{ \mu _{A}(e),1-\mu _{A}(e)\right\}
\right\vert },\frac{\left\vert \min \left\{
f_{A}(e)(x),1-f_{A}(e)(x)\right\} \right\vert }{\left\vert \max \left\{
f_{A}(e)(x),1-f_{A}(e)(x)\right\} \right\vert }\right) $

$=\left( \frac{\left\vert \min \left\{ 1-(1-\mu _{A}(e)),1-\mu
_{A}(e)\right\} \right\vert }{\left\vert \max \left\{ 1-(1-\mu
_{A}(e)),1-\mu _{A}(e)\right\} \right\vert },\frac{\left\vert \min \left\{
1-(1-f_{A}(e)(x)),1-f_{A}(e)(x)\right\} \right\vert }{\left\vert \max
\left\{ 1-(1-f_{A}(e)(x),1-f_{A}(e)(x)\right\} \right\vert }\right) $

$=\left( \frac{\pi \left( \left( A^{c}\right) ^{c}\cap A^{c}\right) }{\pi
\left( \left( A^{c}\right) ^{c}\cap A^{c}\right) },\frac{\sigma \left(
\left( F^{c}\right) ^{c}\cap F^{c}\right) }{\sigma \left( \left(
F^{c}\right) ^{c}\cap F^{c}\right) }\right) =\left( \frac{\pi \left(
A^{c}\cap \left( A^{c}\right) ^{c}\right) }{\pi \left( A^{c}\cup \left(
A^{c}\right) ^{c}\right) },\frac{\sigma \left( F^{c}\cap \left( F^{c}\right)
^{c}\right) }{\sigma \left( F^{c}\cup \left( F^{c}\right) ^{c}\right) }%
\right) =ent(\varepsilon ,\kappa )F_{A}^{c}.$

$(4)$ Consider $\mu _{A}(e)\leq \mu _{B}(e)\leq 0.5$ and $f_{A}(e)(x)\leq
g_{B}(e)(x)\leq 0.5$ or $0.5\leq \mu _{B}(e)\leq \mu _{A}(e)$ and $0.5\leq
g_{B}(e)(x)\leq f_{A}(e)(x)$ for all $e\in E$ and all $x\in U.$ Then

$ent(\varepsilon ,\kappa )(F_{A})=\left( \varepsilon (A),\kappa (F)\right)
=\left( \frac{\pi \left( A\cap A^{c}\right) }{\pi \left( A\cup A^{c}\right) }%
,\frac{\sigma \left( F\cap F^{c}\right) }{\sigma \left( F\cup F^{c}\right) }%
\right) $

$=\left( \frac{\left\vert \min \left\{ \mu _{A}(e),1-\mu _{A}(e)\right\}
\right\vert }{\left\vert \max \left\{ \mu _{A}(e),1-\mu _{A}(e)\right\}
\right\vert },\frac{\left\vert \min \left\{
f_{A}(e)(x),1-f_{A}(e)(x)\right\} \right\vert }{\left\vert \max \left\{
f_{A}(e)(x),1-f_{A}(e)(x)\right\} \right\vert }\right) \leq \left( \frac{%
\left\vert \min \left\{ \mu _{B}(e),1-\mu _{B}(e)\right\} \right\vert }{%
\left\vert \max \left\{ \mu _{B}(e),1-\mu _{B}(e)\right\} \right\vert },%
\frac{\left\vert \min \left\{ g_{B}(e)(x),1-g_{B}(e)(x)\right\} \right\vert 
}{\left\vert \max \left\{ g_{B}(e)(x),1-g_{B}(e)(x)\right\} \right\vert }%
\right) $

$=\left( \frac{\pi \left( B\cap B^{c}\right) }{\pi \left( B\cup B^{c}\right) 
},\frac{\sigma \left( G\cap G^{c}\right) }{\sigma \left( G\cup G^{c}\right) }%
\right) =ent(\varepsilon ,\kappa )(G_{B}).$

\begin{remark}
Let $F_{A}$ be a soft hybrid set over $U$ , i.e., $F_{A}\in \mathcal{S}(U).$
Then it is clear that $ent(\varepsilon ,\kappa )(F_{A})=(0,0)=0.$
\end{remark}

\begin{definition}
Let $F_{A}$ be a fuzzy parameterized soft set over $U$ , i.e., $F_{A}\in
X(U).$
\end{definition}

(1) The function $ent(\varepsilon ,\kappa )$ defined by

$ent(\varepsilon ,\kappa )(F_{A})=\left( \varepsilon (A),\kappa (F)\right)
=\left( \frac{\pi \left( A\cap A^{c}\right) }{\pi \left( A\cup A^{c}\right) }%
,\frac{\sigma \left( F\cap F^{c}\right) }{\sigma \left( F\cup F^{c}\right) }%
\right) =\left( \frac{\left\vert A\cap A^{c}\right\vert }{\left\vert A\cup
A^{c}\right\vert },0\right) =\left( \frac{\tsum\nolimits_{e\in A}\min \left(
\mu _{A}(e),1-\mu _{A}(e)\right) }{\tsum\nolimits_{e\in A}\max \left( \mu
_{A}(e),1-\mu _{A}(e)\right) },0\right) ,$ where $\mu _{A}(e)\in \mathcal{F}%
(E)$ for all $e\in E,$ is called the cardinality of fuzzy parameterized soft
set $F_{A}.$

(2) The function $ent(\varepsilon ,\kappa )$ defined by

$ent(\varepsilon ,\kappa )(F_{A})=\left( \varepsilon (A),\kappa (F)\right)
=\left( \frac{\pi \left( A\cap A^{c}\right) }{\pi \left( A\cup A^{c}\right) }%
,\frac{\sigma \left( F\cap F^{c}\right) }{\sigma \left( F\cup F^{c}\right) }%
\right) =\left( 0,\frac{\tsum\nolimits_{e\in A}\tsum\nolimits_{x\in U}\min
\left( f_{A}(e)(x),1-f_{A}(e\right) (x)}{\tsum\nolimits_{e\in
A}\tsum\nolimits_{x\in U}\max \left( f_{A}(e)(x),1-f_{A}(e\right) (x)}%
\right) ,$ where $f_{A}(e)\in \mathcal{F}(U)$ for all $e\in E,$ is called
the cardinality of fuzzy soft set $F_{A}.$

(3) The function $ent(\varepsilon ,\kappa )$ defined by $ent(\varepsilon
,\kappa )(F_{A})=\left( \varepsilon (A),\kappa (F)\right) =\left( \frac{%
\left\vert A\cap A^{c}\right\vert }{\left\vert A\cup A^{c}\right\vert },%
\frac{\left\vert F\cap F^{c}\right\vert }{\left\vert F\cup F^{c}\right\vert }%
,\right) $

$=\left( \frac{\tsum\nolimits_{e\in A}\min \left( \mu _{A}(e),1-\mu
_{A}(e)\right) }{\tsum\nolimits_{e\in A}\max \left( \mu _{A}(e),1-\mu
_{A}(e)\right) },\frac{\tsum\nolimits_{e\in A}\tsum\nolimits_{x\in U}\min
\left( f_{A}(e)(x),1-f_{A}(e\right) (x)}{\tsum\nolimits_{e\in
A}\tsum\nolimits_{x\in U}\max \left( f_{A}(e)(x),1-f_{A}(e\right) (x)}%
\right) ,$ where $f_{A}(e)\in \mathcal{F}(U)$ and $\mu _{A}(e)\in \mathcal{F}%
(E)$ for all $e\in E$, is called the cardinality of fuzzy parameterized
fuzzy soft set $F_{A}.$

\begin{example}
Consider the soft hybrid sets $\left( G_{B}\right) _{s}$, $\left(
G_{B}\right) _{fps}$, $\left( G_{B}\right) _{fs}$ and $\left( G_{B}\right)
_{fpfs}.$ Then
\end{example}

$ent(\varepsilon ,\kappa )\left( G_{B}\right) _{s}=\left( 0,0\right) ,$ $%
ent(\varepsilon ,\kappa )\left( G_{B}\right) _{fps}=\left( 0.42,0\right) ,$ $%
ent(\varepsilon ,\kappa )\left( G_{B}\right) _{fs}=\left( 0,0.42\right) ,$

$ent(\varepsilon ,\kappa )\left( G_{B}\right) _{fpfs}=\left(
0.50,0.48\right) .$

\begin{theorem}
Let $F_{A}$ and $G_{B}$ be two soft hybrid sets over $U$ , i.e., $%
F_{A},G_{B}\in X(U)$. Then

\begin{enumerate}
\item $ent(\varepsilon ,\kappa )(F_{A})+ent(\varepsilon ,\kappa )\left(
G_{B}\right) =ent(\varepsilon ,\kappa )\left( F_{A}\tilde{\cap}G_{B}\right)
+ent(\varepsilon ,\kappa )\left( F_{A}\tilde{\cup}G_{B}\right) .$

\item $ent(\varepsilon ,\kappa )\left( F_{A\times B}\right) +ent(\varepsilon
,\kappa )\left( G_{A\times B}\right) =ent(\varepsilon ,\kappa )\left( F_{A}%
\tilde{\wedge}G_{B}\right) +ent(\varepsilon ,\kappa )\left( F_{A}\tilde{\vee}%
G_{B}\right) .$
\end{enumerate}

\begin{proof}
Let $F_{A}$ and $G_{B}$ be two soft hybrid sets over $U$ , i.e., $%
F_{A},G_{B}\in X(U)$. For all $e\in E$ and all $x\in U,$
\end{proof}
\end{theorem}

$(1)$ $ent(\varepsilon ,\kappa )\left( F_{A}\tilde{\cap}G_{B}\right)
+ent(\varepsilon ,\kappa )\left( F_{A}\tilde{\cup}G_{B}\right) $

\bigskip $=\left( \frac{\left\vert \left( A\cap B\right) \cap \left(
A^{c}\cup B^{c}\right) \right\vert }{\left\vert \left( A\cap B\right) \cup
\left( A^{c}\cup B^{c}\right) \right\vert },\frac{\left\vert \left( F\cap
G\right) \cap \left( F^{c}\cup G^{c}\right) \right\vert }{\left\vert \left(
F\cap G\right) \cup \left( F^{c}\cup G^{c}\right) \right\vert }\right)
+\left( \frac{\left\vert \left( A\cup B\right) \cap \left( A^{c}\cap
B^{c}\right) \right\vert }{\left\vert \left( A\cup B\right) \cup \left(
A^{c}\cap B^{c}\right) \right\vert },\frac{\left\vert \left( F\cup G\right)
\cap \left( F^{c}\cap G^{c}\right) \right\vert }{\left\vert \left( F\cup
G\right) \cup \left( F^{c}\cap G^{c}\right) \right\vert }\right) $

$=\left( \frac{\left\vert \min \left( \min \left\{ \mu _{A}(e),\mu
_{B}(e)\right\} ,\max \left\{ 1-\mu _{A}(e),1-\mu _{B}(e)\right\} \right)
\right\vert }{\left\vert \max \left( \min \left\{ \mu _{A}(e),\mu
_{B}(e)\right\} ,\max \left\{ 1-\mu _{A}(e),1-\mu _{B}(e)\right\} \right)
\right\vert }\right. ,\left. \frac{\left\vert \min \left( \min \left\{
f_{A}(e)(x),g_{B}(e)(x)\right\} ,\max \left\{
1-f_{A}(e)(x),1-g_{B}(e)(x)\right\} \right) \right\vert }{\left\vert \max
\left( \min \left\{ f_{A}(e)(x),g_{B}(e)(x)\right\} ,\max \left\{
1-f_{A}(e)(x),1-g_{B}(e)(x)\right\} \right) \right\vert }\right) $

$+\left( \frac{\left\vert \min \left( \max \left\{ \mu _{A}(e),\mu
_{B}(e)\right\} ,\min \left\{ 1-\mu _{A}(e),1-\mu _{B}(e)\right\} \right)
\right\vert }{\left\vert \max \left( \max \left\{ \mu _{A}(e),\mu
_{B}(e)\right\} ,\min \left\{ 1-\mu _{A}(e),1-\mu _{B}(e)\right\} \right)
\right\vert }\right. ,\left. \frac{\left\vert \min \left( \max \left\{
f_{A}(e)(x),g_{B}(e)(x)\right\} ,\min \left\{
1-f_{A}(e)(x),1-g_{B}(e)(x)\right\} \right) \right\vert }{\left\vert \max
\left( \max \left\{ f_{A}(e)(x),g_{B}(e)(x)\right\} ,\min \left\{
1-f_{A}(e)(x),1-g_{B}(e)(x)\right\} \right) \right\vert }\right) $

$=\left( \frac{\left\vert \min \left( \min \left\{ \mu _{A}(e),\mu
_{B}(e)\right\} ,\max \left\{ 1-\mu _{A}(e),1-\mu _{B}(e)\right\} \right)
\right\vert }{\left\vert \max \left( \min \left\{ \mu _{A}(e),\mu
_{B}(e)\right\} ,\max \left\{ 1-\mu _{A}(e),1-\mu _{B}(e)\right\} \right)
\right\vert }+\frac{\left\vert \min \left( \max \left\{ \mu _{A}(e),\mu
_{B}(e)\right\} ,\min \left\{ 1-\mu _{A}(e),1-\mu _{B}(e)\right\} \right)
\right\vert }{\left\vert \max \left( \max \left\{ \mu _{A}(e),\mu
_{B}(e)\right\} ,\min \left\{ 1-\mu _{A}(e),1-\mu _{B}(e)\right\} \right)
\right\vert }\right) $

$,\left( \frac{\left\vert \min \left( \min \left\{
f_{A}(e)(x),g_{B}(e)(x)\right\} ,\max \left\{
1-f_{A}(e)(x),1-g_{B}(e)(x)\right\} \right) \right\vert }{\left\vert \max
\left( \min \left\{ f_{A}(e)(x),g_{B}(e)(x)\right\} ,\max \left\{
1-f_{A}(e)(x),1-g_{B}(e)(x)\right\} \right) \right\vert }\right. $

$+\left. \frac{\left\vert \min \left( \max \left\{
f_{A}(e)(x),g_{B}(e)(x)\right\} ,\min \left\{
1-f_{A}(e)(x),1-g_{B}(e)(x)\right\} \right) \right\vert }{\left\vert \max
\left( \max \left\{ f_{A}(e)(x),g_{B}(e)(x)\right\} ,\min \left\{
1-f_{A}(e)(x),1-g_{B}(e)(x)\right\} \right) \right\vert }\right) $

$=\left( \frac{\left\vert \min \left\{ \mu _{A}(e),1-\mu _{A}(e)\right\}
\right\vert }{\left\vert \max \left\{ \mu _{A}(e),1-\mu _{A}(e)\right\}
\right\vert }+\frac{\left\vert \min \left\{ \mu _{B}(e),1-\mu
_{B}(e)\right\} \right\vert }{\left\vert \max \left\{ \mu _{B}(e),1-\mu
_{B}(e)\right\} \right\vert }\right) ,\left( \frac{\left\vert \min \left\{
f_{A}(e)(x),1-f_{A}(e)\right\} \right\vert }{\left\vert \max \left\{
f_{A}(e)(x),1-f_{A}(e)\right\} \right\vert }+\frac{\left\vert \min \left\{
g_{B}(e)(x),1-g_{B}(e)(x)\right\} \right\vert }{\left\vert \max \left\{
g_{B}(e)(x),1-g_{B}(e)(x)\right\} \right\vert }\right) $

$=\left( \frac{\left\vert \min \left\{ \mu _{A}(e),1-\mu _{A}(e)\right\}
\right\vert }{\left\vert \max \left\{ \mu _{A}(e),1-\mu _{A}(e)\right\}
\right\vert },\frac{\left\vert \min \left\{
f_{A}(e)(x),1-f_{A}(e)(x)\right\} \right\vert }{\left\vert \max \left\{
f_{A}(e)(x),1-f_{A}(e)(x)\right\} \right\vert }\right) +\left( \frac{%
\left\vert \min \left\{ \mu _{B}(e),1-\mu _{B}(e)\right\} \right\vert }{%
\left\vert \max \left\{ \mu _{B}(e),1-\mu _{B}(e)\right\} \right\vert },%
\frac{\left\vert \min \left\{ g_{B}(e)(x),1-g_{B}(e(x))\right\} \right\vert 
}{\left\vert \max \left\{ g_{B}(e)(x),1-g_{B}(e)(x)\right\} \right\vert }%
\right) $

$=\left( \frac{\left\vert A\cap A^{c}\right\vert }{\left\vert A\cup
A^{c}\right\vert },\frac{\left\vert F\cap F^{c}\right\vert }{\left\vert
F\cup F^{c}\right\vert }\right) +\left( \frac{\left\vert B\cap
B^{c}\right\vert }{\left\vert B\cup B^{c}\right\vert },\frac{\left\vert
G\cap G^{c}\right\vert }{\left\vert G\cup G^{c}\right\vert }\right)
=ent(\varepsilon ,\kappa )(F_{A})+ent(\varepsilon ,\kappa )(G_{B}).$

\bigskip $(2)$ Similarly, we can easily get that $ent(\varepsilon ,\kappa
)\left( F_{A\times B}\right) +ent(\varepsilon ,\kappa )\left( G_{A\times
B}\right) =ent(\varepsilon ,\kappa )\left( F_{A}\tilde{\wedge}G_{B}\right)
+ent(\varepsilon ,\kappa )\left( F_{A}\tilde{\vee}G_{B}\right) .$

\section{Similarity of soft hybrid sets}

\begin{definition}
\label{benzerlik def}Let $F_{A}$ and $G_{B}$ be two soft hybrid sets over $U$
, i.e., $F_{A},G_{B}\in X(U)$. Let $sim(\alpha ,\beta )$ be a mapping given
by $sim(\alpha ,\beta ):$ $X(U)\times X(U)\longrightarrow \left[ 0,1\right]
\times \left[ 0,1\right] $, where $\alpha :P(E)\times P(E)\longrightarrow %
\left[ 0,1\right] $ (or $\mathcal{F}(E)\times \mathcal{F}(E)\longrightarrow %
\left[ 0,1\right] )$ and $\beta :P(U)$ $\times P(U)\longrightarrow \left[ 0,1%
\right] $ (or $\mathcal{F}(U)\times \mathcal{F}(U)\longrightarrow \left[ 0,1%
\right] )$ are two mappings. Then $sim(\alpha ,\beta )$ is called the
similarity measure of $F_{A}$ and $G_{B}$ is defined by%
\begin{equation*}
sim(\alpha ,\beta )\left( F_{A},G_{B}\right) =\left( \alpha (A,B),\beta
(F,G)\right)
\end{equation*}%
where $\alpha (A,B)$ is the similarity measure of $A$ and $B$ while $\beta
(F,G)$ is the similarity measure of $f_{A}(e)$ and $f_{B}(e)$ $\in E$ for
all $e\in E,$ if it satisfies the following axiomatic reguirements:

\begin{enumerate}
\item $0\leq sim(\alpha ,\beta )\left( F_{A},G_{B}\right) \leq 1.$

\item $sim(\alpha ,\beta )\left( F_{A},G_{B}\right) =1$ if and only if $%
F_{A}=G_{B}$, i.e., $A=B$ and $F\left( e\right) =G\left( e\right) $ for all $%
e\in E.$

\item $sim(\alpha ,\beta )\left( F_{A},G_{B}\right) =sim(\alpha ,\beta
)\left( G_{B},F_{A}\right) .$

\item If $F_{A}$ $\tilde{\subseteq}$ $G_{B}\tilde{\subseteq}H_{C}$, then $%
sim(\alpha ,\beta )\left( F_{A},H_{C}\right) \leq sim(\alpha ,\beta )\left(
F_{A},G_{B}\right) $ and $sim(\alpha ,\beta )\left( F_{A},H_{C}\right) \leq
sim(\alpha ,\beta )\left( G_{B},H_{C}\right) .$
\end{enumerate}
\end{definition}

\begin{definition}
Let $F_{A}$ and $G_{B}$ be two soft hybrid sets over $U$ , i.e., $%
F_{A},G_{B}\in X(U)$. Then 
\begin{equation*}
sim(\alpha ,\beta )\left( F_{A},G_{B}\right) =\left( \alpha (A,B),\beta
(F,G)\right) =\left( \frac{\pi \left( A\cap B\right) }{\pi \left( A\cup
B\right) },\frac{\sigma \left( F\cap G\right) }{\sigma \left( F\cup G\right) 
}\right)
\end{equation*}%
where mapping $count(\pi ,\sigma )$ is a cardinal function.
\end{definition}

\begin{theorem}
The above-defined measure $sim(\alpha ,\beta )\left( F_{A},G_{B}\right) $
for soft hybrid sets is a similarity measure for soft hybrid sets over $U$,
i.e., it is satisfies all the properties in Definition \ref{benzerlik def}.
\end{theorem}

(1) It is clear.

(2) $sim(\alpha ,\beta )\left( F_{A},G_{B}\right) =1\Longleftrightarrow
sim(\alpha ,\beta )\left( F_{A},G_{B}\right) =\left( \alpha (A,B),\beta
(F,G)\right) =\left( \frac{\pi \left( A\cap B\right) }{\pi \left( A\cup
B\right) },\frac{\sigma \left( F\cap G\right) }{\sigma \left( F\cup G\right) 
}\right) =1\Longleftrightarrow \frac{\pi \left( A\cap B\right) }{\pi \left(
A\cup B\right) }=1$ and $\frac{\sigma \left( F\cap G\right) }{\sigma \left(
F\cup G\right) }=1\Longleftrightarrow \pi \left( A\cap B\right) =\pi \left(
A\cup B\right) $ and $\sigma \left( F\cap F\right) =\sigma \left( F\cup
F\right) \Longleftrightarrow $ $\left\vert A\cap B\right\vert =\left\vert
A\cup B\right\vert $ and $\left\vert F\cap F\right\vert =\left\vert F\cup
F\right\vert \Longleftrightarrow \left\vert \min \left\{ \mu _{A}(e),\mu
_{B}(e)\right\} \right\vert =\left\vert \max \left\{ \mu _{A}(e),\mu
_{B}(e)\right\} \right\vert $ and $\left\vert \min \left\{
f_{A}(e)(x),g_{B}(e)(x)\right\} \right\vert =\left\vert \max \left\{
f_{A}(e)(x),g_{B}(e)(x)\right\} \right\vert $ for all $e\in E$ and all $x\in
U$ $\Longleftrightarrow \mu _{A}(e)=\mu _{B}(e)$ and $%
f_{A}(e)(x)=g_{B}(e)(x) $ for all $e\in E$ and all $x\in U.$ That is, $%
F_{A}=G_{B}.$

(3) $sim(\alpha ,\beta )\left( F_{A},G_{B}\right) =\left( \alpha (A,B),\beta
(F,G)\right) =\left( \frac{\pi \left( A\cap B\right) }{\pi \left( A\cup
B\right) },\frac{\sigma \left( F\cap G\right) }{\sigma \left( F\cup G\right) 
}\right) =\left( \frac{\pi \left( B\cap A\right) }{\pi \left( B\cup A\right) 
},\frac{\sigma \left( G\cap F\right) }{\sigma \left( G\cup F\right) }\right)
=sim(\alpha ,\beta )\left( G_{B},F_{A}\right) .$

Let $F_{A}$ $\tilde{\subseteq}$ $G_{B}\tilde{\subseteq}H_{C}$. Then we have $%
\mu _{A}(e)\leq \mu _{B}(e)\leq \mu _{C}(e)$ and $f_{A}(e)(x)\leq
g_{B}(e)(x)\leq h_{C}$ $(e)(x)$ for all $e\in E$ and all $x\in U$. Then

(4) $sim(\alpha ,\beta )\left( F_{A},H_{C}\right) =\left( \alpha (A,C),\beta
(F,H)\right) =\left( \frac{\pi \left( A\cap C\right) }{\pi \left( A\cup
C\right) },\frac{\sigma \left( F\cap H\right) }{\sigma \left( F\cup H\right) 
}\right) =\left( \frac{\left\vert A\cap C\right\vert }{\left\vert A\cup
C\right\vert },\frac{\left\vert F\cap H\right\vert }{\left\vert F\cup
H\right\vert }\right) $

$=\left( \frac{\left\vert \min \left\{ \mu _{A}(e),\mu _{C}(e)\right\}
\right\vert }{\left\vert \max \left\{ \mu _{A}(e),\mu _{C}(e)\right\}
\right\vert },\frac{\left\vert \min \left\{ f_{A}(e)(x),h_{C}(e)(x)\right\}
\right\vert }{\left\vert \max \left\{ f_{A}(e)(x),h_{C}(e)(x)\right\}
\right\vert }\right) \leq \left( \frac{\left\vert \min \left\{ \mu
_{A}(e),\mu _{B}(e)\right\} \right\vert }{\left\vert \max \left\{ \mu
_{A}(e),\mu _{B}(e)\right\} \right\vert },\frac{\left\vert \min \left\{
f_{A}(e)(x),g_{B}(e))(x)\right\} \right\vert }{\left\vert \max \left\{
f_{A}(e)(x),g_{B}(e)(x)\right\} \right\vert }\right) $

$=\left( \frac{\pi \left( A\cap B\right) }{\pi \left( A\cup B\right) },\frac{%
\sigma \left( F\cap G\right) }{\sigma \left( F\cup G\right) }\right)
=sim(\alpha ,\beta )\left( F_{A},G_{B}\right) $. So $sim(\alpha ,\beta
)\left( F_{A},H_{C}\right) \leq sim(\alpha ,\beta )\left( F_{A},G_{B}\right)
.$

Similarly$,$ we can see that $sim(\alpha ,\beta )\left( F_{A},H_{C}\right)
\leq sim(\alpha ,\beta )\left( G_{B},H_{C}\right) .$

\begin{definition}
Let $F_{A}$ and $G_{B}$ be two soft hybrid sets over $U$ , i.e., $%
F_{A},G_{B}\in X(U).$
\end{definition}

(1) The function $sim(\alpha ,\beta )$ defined by $sim(\alpha ,\beta )\left(
F_{A},G_{B}\right) =\left( \alpha (A,B),\beta (F,G)\right) $

$=\left( \frac{\left\vert A\cap B\right\vert }{\left\vert A\cup B\right\vert 
},\frac{\left\vert f_{A}(e)\cap g_{B}(e)\right\vert }{\left\vert
f_{A}(e)\cap g_{B}(e)\right\vert }\right) $ where $f_{A}(e),g_{B}(e)\in P(U)$
for each $e\in E,$ is a similarity measure of soft sets $F_{A}$ and $G_{B}.$

(2) The function $sim(\alpha ,\beta )$ defined by

$sim(\alpha ,\beta )\left( F_{A},G_{B}\right) =\left( \alpha (A,B),\beta
(F,G)\right) =\left( \frac{\tsum\nolimits_{e\in A\cap B}\min \left\{ \mu
_{A}(e),\mu _{B}(e)\right\} }{\tsum\nolimits_{e\in A\cap B}\max \left\{ \mu
_{A}(e),\mu _{B}(e)\right\} },\frac{\left\vert f_{A}(e)\cap
g_{B}(e)\right\vert }{\left\vert f_{A}(e)\cap g_{B}(e)\right\vert }\right) $
where $\mu _{A}(e),\mu _{B}(e)\in \mathcal{F}(E)$ and $f_{A}(e),g_{B}(e)\in
P(U)$ for all $e\in E,$ is a similarity measure of fuzzy parameterized soft
sets $F_{A}$ and $G_{B}.$

(3) The function $sim(\alpha ,\beta )$ defined by

$sim(\alpha ,\beta )\left( F_{A},G_{B}\right) =\left( \alpha (A,B),\beta
(F,G)\right) =\left( \frac{\left\vert A\cap B\right\vert }{\left\vert A\cup
B\right\vert },\frac{\tsum\nolimits_{e\in A\cap B}\tsum\nolimits_{x\in
U}\min \left\{ f_{A}(e)(x),g_{B}(e)(x)\right\} }{\tsum\nolimits_{e\in A\cap
B}\tsum\nolimits_{x\in U}\max \left\{ f_{A}(e)(x),g_{B}(e)(x)\right\} }%
\right) $ where $f_{A}(e),g_{B}(e)\in \mathcal{F}(U)$ for all $e\in E,$ is a
similarity measure of fuzzy soft sets $F_{A}$ and $G_{B}.$

(4) The function $sim(\alpha ,\beta )$ defined by

$sim(\alpha ,\beta )\left( F_{A},G_{B}\right) =\left( \alpha (A,B),\beta
(F,G)\right) $

$=\left( \frac{\tsum\nolimits_{e\in A\cap B}\min \left\{ \mu _{A}(e),\mu
_{B}(e)\right\} }{\tsum\nolimits_{e\in A\cap B}\max \left\{ \mu _{A}(e),\mu
_{B}(e)\right\} },\frac{\tsum\nolimits_{e\in A\cap B}\tsum\nolimits_{x\in
U}\min \left\{ f_{A}(e)(x),g_{B}(e)(x)\right\} }{\tsum\nolimits_{e\in A\cap
B}\tsum\nolimits_{x\in U}\max \left\{ f_{A}(e)(x),g_{B}(e)(x)\right\} }%
\right) $\textsc{\ }where $\mu _{A}(e),\mu _{B}(e)\in \mathcal{F}(E)$ and $%
f_{A}(e),g_{B}(e)\in \mathcal{F}(U)$ for all $e\in E,$ is a similarity
measure of fuzzy parameterized fuzzy soft sets $F_{A}$ and $G_{B}.$

\begin{example}
Consider the soft hybrid sets $(F_{A})_{s}$,$(F_{A})_{fps}$, $(F_{A})_{fs}$,$%
(F_{A})_{fpfs}$ and $\left( G_{B}\right) _{s}$, $\left( G_{B}\right) _{fps}$%
, $\left( G_{B}\right) _{fs}$,$\left( G_{B}\right) _{fpfs}$. Then%
\begin{equation*}
\begin{array}[t]{ll}
sim(\alpha ,\beta )\left( (F_{A})_{s},\left( G_{B}\right) _{s}\right)
=\left( 0.75,0.38\right) & sim(\alpha ,\beta )\left( (F_{A})_{fps}\left(
G_{B}\right) _{fps}\right) =\left( 0.38,0.25\right) \\ 
sim(\alpha ,\beta )\left( (F_{A})_{fs},\left( G_{B}\right) _{fs}\right)
=\left( 0.60,0.15\right) & sim(\alpha ,\beta )\left( (F_{A})_{fpfs},\left(
G_{B}\right) _{fpfs}\right) =\left( 0.50,0.20\right) .%
\end{array}%
\end{equation*}
\end{example}

\begin{theorem}
Let $F_{A}$ be a soft hybrid set over $U$ , i.e., $F_{A}\in X(U)$. Then $%
ent(\varepsilon ,\kappa )(F_{A})=sim(\alpha ,\beta )\left( F_{A}\tilde{\cup}%
F_{A}^{c},F_{A}\tilde{\cap}F_{A}^{c}\right) .$

\begin{proof}
$sim(\alpha ,\beta )\left( F_{A}\tilde{\cap}F_{A}^{c},F_{A}\tilde{\cup}%
F_{A}^{c}\right) =\left( \frac{\pi \left( \left( A\cap A^{c}\right) \cap
\left( A\cup A^{c}\right) \right) }{\pi \left( \left( A\cap A^{c}\right)
\cup \left( A\cup A^{c}\right) \right) },\frac{\sigma \left( \left( F\cap
F^{c}\right) \cap \left( F\cup F^{c}\right) \right) }{\sigma \left( \left(
F\cap F^{c}\right) \cup \left( F\cup F^{c}\right) \right) }\right) $

$=\left( \frac{\left\vert \left( A\cap A^{c}\right) \cap \left( A\cup
A^{c}\right) \right\vert }{\left\vert \left( A\cap A^{c}\right) \cup \left(
A\cup A^{c}\right) \right\vert },\frac{\left\vert \left( F\cap F^{c}\right)
\cap \left( F\cup F^{c}\right) \right\vert }{\left\vert \left( F\cap
F^{c}\right) \cup \left( F\cup F^{c}\right) \right\vert }\right) =\left( 
\frac{\pi \left( A\cap A^{c}\right) }{\pi \left( A\cup A^{c}\right) },\frac{%
\sigma \left( F\cap F^{c}\right) }{\sigma \left( F\cup F^{c}\right) }\right) 
$

$=\left( \frac{\left\vert \min \left\{ \mu _{A}(e),1-\mu _{A}(e)\right\}
\right\vert }{\left\vert \max \left\{ \mu _{A}(e),1-\mu _{A}(e)\right\}
\right\vert },\frac{\left\vert \min \left\{
f_{A}(e)(x),1-f_{A}(e)(x)\right\} \right\vert }{\left\vert \max \left\{
f_{A}(e)(x),1-f_{A}(e)(x)\right\} \right\vert }\right) =ent(\varepsilon
,\kappa )(F_{A}).$
\end{proof}
\end{theorem}

\section{Subsethood of soft hybrid sets}

\begin{definition}
\label{subsethood def}Let $F_{A}$ and $G_{B}$ be two soft hybrid sets over $%
U $ , i.e., $F_{A},G_{B}\in X(U)$. Let $sub(\theta ,\delta )$ be a mapping $%
sub(\theta ,\delta ):$ $X(U)\times X(U)\longrightarrow \left[ 0,1\right]
\times \left[ 0,1\right] $ where $\theta :P(E)\times P(E)\longrightarrow %
\left[ 0,1\right] $ (or $\theta :\mathcal{F}(E)\times \mathcal{F}%
(E)\longrightarrow \left[ 0,1\right] )$ and $\delta :P(U)$ $\times
P(U)\longrightarrow \left[ 0,1\right] $ (or $\delta :\mathcal{F}(U)\times 
\mathcal{F}(U)\longrightarrow \left[ 0,1\right] )$ are two mappings. Then $%
sub(\theta ,\delta )\left( F_{A},G_{B}\right) $ is called the subsethood
measure of $F_{A}$ and $G_{B}$ is defined by%
\begin{equation*}
sub(\theta ,\delta )\left( F_{A},G_{B}\right) =\left( \theta (A,B),\delta
(F,G)\right)
\end{equation*}%
where $\theta (A,B)$ is the subsethood measure of $A$ and $B$ while $\delta
(F,G)$ is the subsethood measure of $f_{A}(e)$ and $f_{B}(e)$ for all $e\in
E,$ if it satisfies the following axiomatic reguirements:
\end{definition}

(1) $sub(\theta ,\delta )\left( F_{A},G_{B}\right) =1$ if and only if $F_{A}%
\tilde{\subseteq}G_{B}.$

(2) Let $F_{A}^{c}\tilde{\subseteq}F_{A}$. Then $sub(\theta ,\delta )\left(
F_{A},F_{A}^{c}\right) =0$ if and only if $F_{A}=\tilde{U}$.

(3) If $F_{A}$ $\tilde{\subseteq}$ $G_{B}\tilde{\subseteq}H_{C}$, then $%
sub(\theta ,\delta )\left( H_{C},F_{A}\right) \leq sub(\theta ,\delta
)\left( G_{B},F_{A}\right) ;$

and if $F_{A}$ $\tilde{\subseteq}$ $G_{B}$, $sub(\theta ,\delta )\left(
K_{D},F_{A}\right) \leq sub(\theta ,\delta )\left( K_{D},G_{B}\right) .$

\begin{definition}
Let $F_{A}$ and $G_{B}$ be two soft hybrid sets over $U$ , i.e., $%
F_{A},G_{B}\in X(U)$. Then 
\begin{equation*}
sub(\theta ,\delta )\left( F_{A},G_{B}\right) =\left( \theta (A,B),\delta
(F,G)\right) =\left( \frac{\pi \left( A\cap B\right) }{\pi \left( A\right) },%
\frac{\sigma \left( F\cap G\right) }{\sigma \left( F\right) }\right)
\end{equation*}%
where mapping $count(\pi ,\sigma )$ is a cardinal function.
\end{definition}

\begin{theorem}
The above-defined measure $sub(\theta ,\delta )\left( F_{A},G_{B}\right) $
is a subsethood measure for soft hybrid sets over $U$, i.e., it is satiffies
all the properties in Definition \ref{subsethood def}.
\end{theorem}

$(1)$ $sub(\theta ,\delta )\left( F_{A},G_{B}\right) =1\Longleftrightarrow
\left( \frac{\pi \left( A\cap B\right) }{\pi \left( A\right) },\frac{\sigma
\left( F\cap G\right) }{\sigma \left( F\right) }\right) =\left( \frac{%
\left\vert A\cap B\right\vert }{\left\vert A\right\vert },\frac{\left\vert
F\cap G\right\vert }{\left\vert F\right\vert }\right) =1\Longleftrightarrow 
\frac{\left\vert A\cap B\right\vert }{\left\vert A\right\vert }=1$ and $%
\frac{\left\vert F\cap G\right\vert }{\left\vert F\right\vert }%
=1\Longleftrightarrow \left\vert A\cap B\right\vert =\left\vert A\right\vert 
$ and $\left\vert F\cap G\right\vert =\left\vert F\right\vert
\Longleftrightarrow \left\vert \min \left\{ \left( \mu _{A}(e),\mu
_{B}(e)\right) \right\} \right\vert =\left\vert \mu _{A}(e)\right\vert $ and 
$\left\vert \min \left( f_{A}(e)(x),g_{B}(e)(x)\right) \right\vert
=\left\vert f_{A}(e)(x)\right\vert $ for all $e\in E$ and all $x\in
U\Longleftrightarrow \mu _{A}(e)\leq \mu _{B}(e)$ and $\left(
f_{A}(e)(x)\leq g_{B}(e)(x)\right) $ for all $e\in E$ and all $x\in
U\Longleftrightarrow F_{A}\tilde{\subseteq}G_{B}.$

(2) Since $F_{A}^{c}\tilde{\subseteq}F_{A}$, $1-\mu _{A}(e)\leq \mu _{A}(e)$
and $1-f_{A}(e)(x)\leq f_{A}(e)(x)$ for all $e\in E$ and all $x\in U.$

$sub(\theta ,\delta )\left( F_{A},F_{A}^{c}\right) =0\Longleftrightarrow
\left( \frac{\pi \left( A\cap A^{c}\right) }{\pi \left( A^{c}\right) },\frac{%
\sigma \left( F\cap F^{c}\right) }{\sigma \left( F^{c}\right) }\right)
=\left( \frac{\left\vert A\cap A^{c}\right\vert }{\left\vert
A^{c}\right\vert },\frac{\left\vert F\cap F^{c}\right\vert }{\left\vert
F^{c}\right\vert }\right) =0\Longleftrightarrow \frac{\left\vert A\cap
A^{c}\right\vert }{\left\vert A^{c}\right\vert }=0$ and $\frac{\left\vert
F\cap F^{c}\right\vert }{\left\vert F^{c}\right\vert }=0\Longleftrightarrow
\left\vert A\cap A^{c}\right\vert =0$ and $\left\vert F\cap F^{c}\right\vert
=0\Longleftrightarrow \left\vert \min \left\{ \mu _{A}(e),1-\mu
_{A}(e)\right\} \right\vert =0$ and $\left\vert \min \left\{
f_{A}(e)(x),1-f_{A}(e)(x)\right\} \right\vert =0$ for all $e\in E$ and all $%
x\in U\Longleftrightarrow \mu _{A}(e)=1$ and $f_{A}(e)(x)=1$ for all $e\in E$
and all $x\in U\Longleftrightarrow F_{A}=\tilde{U}.$

Let $F_{A}$ $\tilde{\subseteq}$ $G_{B}\tilde{\subseteq}H_{C}$. Then we have $%
\mu _{A}(e)\leq \mu _{B}(e)\leq \mu _{C}(e)$ and $f_{A}(e)(x)\leq
g_{B}(e)(x)\leq h_{C}$ $(e)(x)$ for all $e\in E$ and all $x\in U$. Then

(3) $sub(\theta ,\delta )\left( H_{C},F_{A}\right) =\left( \frac{\pi \left(
C\cap A\right) }{\pi \left( C\right) },\frac{\sigma \left( H\cap F\right) }{%
\sigma \left( H\right) }\right) =\left( \frac{\left\vert C\cap A\right\vert 
}{\left\vert C\right\vert },\frac{\left\vert H\cap F\right\vert }{\left\vert
H\right\vert }\right) $

$=\left( \frac{\left\vert \min \left\{ \mu _{C}(e),\mu _{A}(e)\right\}
\right\vert }{\left\vert \mu _{C}(e)\right\vert },\frac{\left\vert \min
\left\{ h_{C}(e)(x),f_{A}(e)(x)\right\} \right\vert }{\left\vert
h_{C}(e)(x)\right\vert }\right) \leq \left( \frac{\left\vert \min \left\{
\mu _{B}(e),\mu _{A}(e)\right\} \right\vert }{\left\vert \mu
_{B}(e)\right\vert },\frac{\left\vert \min \left\{
g_{B}(e))(x),f_{A}(e)(x)\right\} \right\vert }{\left\vert
g_{B}(e)(x)\right\vert }\right) =\left( \frac{\pi \left( B\cap A\right) }{%
\pi \left( B\right) },\frac{\sigma \left( G\cap F\right) }{\sigma \left(
G\right) }\right) =sub(\theta ,\delta )\left( G_{B},F_{A}\right) $.

So $sub(\theta ,\delta )\left( H_{C},F_{A}\right) \leq sub(\theta ,\delta
)\left( G_{B},F_{A}\right) .$

Similarly$,$ we can see that $sub(\theta ,\delta )\left( K_{D},F_{A}\right)
\leq sub(\theta ,\delta )\left( K_{D},G_{B}\right) $ if $F_{A}$ $\tilde{%
\subseteq}$ $G_{B}$.

\begin{definition}
Let $F_{A}$ and $G_{B}$ be two soft hybrid sets over $U$ , i.e., $%
F_{A},G_{B}\in X(U).$
\end{definition}

(1) The function $sub(\theta ,\delta )$ defined by

$sub(\theta ,\delta )\left( F_{A},G_{B}\right) =\left( \theta (A,B),\delta
(F,G)\right) =\left( \frac{\left\vert A\cap B\right\vert }{\left\vert
A\right\vert },\frac{\left\vert f_{A}(e)\cap g_{B}(e)\right\vert }{%
\left\vert f_{A}(e)\right\vert }\right) ,$ where $f_{A}(e),g_{B}(e)\in P(U)$
for all $e\in E,$ is a subsethood measure of soft sets $F_{A}$ and $G_{B}.$

(2) The function $sub(\theta ,\delta )$ defined by

$sub(\theta ,\delta )\left( F_{A},G_{B}\right) =\left( \theta (A,B),\delta
(F,G)\right) =\left( \frac{\tsum\nolimits_{e\in A\cap B}\min \left( \mu
_{A}(e),\mu _{B}(e)\right) }{\tsum\nolimits_{e\in A\cap B}\mu _{A}(e)},\frac{%
\left\vert f_{A}(e)\cap g_{B}(e)\right\vert }{\left\vert f_{A}(e)\right\vert 
}\right) ,$ where $f_{A}(e),g_{B}(e)\in P(U)$ and $\mu _{A}(e),\mu
_{B}(e)\in \mathcal{F}(E)$ for all $e\in E,$ is a subsethood measure of
fuzzy parameterized soft sets $F_{A}$ and $G_{B}.$

(3) The function $sub(\theta ,\delta )$ defined by

$sub(\theta ,\delta )\left( F_{A},G_{B}\right) =\left( \theta (A,B),\delta
(F,G)\right) =\left( \frac{\left\vert A\cap B\right\vert }{\left\vert
A\right\vert },\frac{\tsum\nolimits_{e\in A\cap B}\tsum\nolimits_{x\in
U}\min \left( f_{A}(e)(x),g_{B}(e)(x)\right) }{\tsum\nolimits_{e\in
A}f_{A}(e)(x)}\right) ,$ where $f_{A}(e),g_{B}(e)\in \mathcal{F}(U)$ for all 
$e\in E,$ is a subsethood measure of fuzzy soft sets $F_{A}$ and $G_{B}.$

(4) The function $sub(\theta ,\delta )$ defined by

$sub(\theta ,\delta )\left( F_{A},G_{B}\right) =\left( \theta (A,B),\delta
(F,G)\right) $

$=\left( \frac{\tsum\nolimits_{e\in A\cap B}\min \left( \mu _{A}(e),\mu
_{B}(e)\right) }{\tsum\nolimits_{e\in A\cap B}\mu _{A}(e)},\frac{%
\tsum\nolimits_{e\in A\cap B}\tsum\nolimits_{x\in U}\min \left(
f_{A}(e)(x),g_{B}(e)(x)\right) }{\tsum\nolimits_{e\in A}f_{A}(e)(x)}\right)
, $ where $f_{A}(e),g_{B}(e)\in \mathcal{F}(U)$ and $\mu _{A}(e),\mu
_{B}(e)\in \mathcal{F}(E)$ for all $e\in E,$ is a similarity measure of
fuzzy parameterized fuzzy soft sets $F_{A}$ and $G_{B}.$

\begin{example}
Consider the soft hybrid sets $(F_{A})_{s}$,$(F_{A})_{fps}$, $(F_{A})_{fs}$,$%
(F_{A})_{fpfs}$ and $\left( G_{B}\right) _{s}$, $\left( G_{B}\right) _{fps}$%
, $\left( G_{B}\right) _{fs}$,$\left( G_{B}\right) _{fpfs}.$ Then

$%
\begin{array}[t]{lll}
sub(\theta ,\delta )\left( (F_{A})_{s},\left( G_{B}\right) _{s}\right)
=\left( 1,1\right) &  & sub(\theta ,\delta )\left( \left( G_{B}\right)
_{s},(F_{A})_{s}\right) =\left( 0.75,0.38\right) \\ 
sub(\theta ,\delta )\left( (F_{A})_{fps},\left( G_{B}\right) _{fps}\right)
=\left( 0.44,0.33\right) &  & sub(\theta ,\delta )\left( \left( G_{B}\right)
_{fps},(F_{A})_{fps}\right) =\left( 0.72,0.60\right) \\ 
sub(\theta ,\delta )\left( (F_{A})_{fs},\left( G_{B}\right) _{fs}\right)
=\left( 1,0.36\right) &  & sub(\theta ,\delta )\left( \left( G_{B}\right)
_{fs},(F_{A})_{fs}\right) =\left( 0.66,0.20\right) \\ 
sub(\theta ,\delta )\left( (F_{A})_{fpfs},\left( G_{B}\right) _{fpfs}\right)
=\left( 1,0.33\right) &  & sub(\theta ,\delta )\left( \left( G_{B}\right)
_{fpfs},(F_{A})_{fpfs}\right) =\left( 0.50,0.33\right)%
\end{array}%
$
\end{example}

Then $sub(\theta ,\delta )\left( (F_{A})_{s},\left( G_{B}\right) _{s}\right)
=1$ and $sub(\theta ,\delta )\left( \left( G_{B}\right)
_{s},(F_{A})_{s}\right) =0$ in Molodtsov's soft subset sense. Here we say
that $(F_{A})_{s}$ is precisely a soft subset of $\left( G_{B}\right) _{s}.$
However, it may be situations being "more and less" a subset of a set in
another set. For example, since $sub(\theta ,\delta )\left(
(F_{A})_{fps},\left( G_{B}\right) _{fps}\right) =\left( 0.44,0.33\right) $
and $sub(\theta ,\delta )\left( \left( G_{B}\right)
_{fps},(F_{A})_{fps}\right) =\left( 0.72,0.66\right) $ , we can say that $%
\left( G_{B}\right) _{fps}$ is much more a soft subset of $(F_{A})_{fps}.$

\begin{theorem}
Let $F_{A}$ and $G_{B}$ be two soft hybrid sets over $U$ , i.e., $%
F_{A},G_{B}\in X(U)$. Then%
\begin{equation*}
sim(\alpha ,\beta )\left( F_{A},G_{B}\right) =sub(\theta ,\delta )\left(
F_{A}\tilde{\cap}G_{B},F_{A}\tilde{\cup}G_{B}\right)
\end{equation*}
\end{theorem}

$sub(\theta ,\delta )\left( F_{A}\tilde{\cap}G_{B},F_{A}\tilde{\cup}%
G_{B}\right) =\left( \theta (A\cap B,A\cup B),\delta (F\cap G,F\cup
G)\right) $

$=\left( \frac{\pi \left( \left( A\cap B\right) \cap \left( A\cup B\right)
\right) }{\pi \left( \left( A\cap B\right) \cup \left( A\cup B\right)
\right) },\frac{\sigma \left( \left( F\cap G\right) \cap \left( F\cup
G\right) \right) }{\sigma \left( \left( F\cap G\right) \cup \left( F\cup
G\right) \right) }\right) =\left( \frac{\left\vert \left( A\cap B\right)
\cap \left( A\cup B\right) \right\vert }{\left\vert \left( A\cap B\right)
\cup \left( A\cup B\right) \right\vert },\frac{\left\vert \left( F\cap
G\right) \cap \left( F\cup G\right) \right\vert }{\left\vert \left( F\cap
G\right) \cup \left( F\cup G\right) \right\vert }\right) =\left( \frac{%
\left\vert A\cap B\right\vert }{\left\vert A\cup B\right\vert },\frac{%
\left\vert F\cap G\right\vert }{\left\vert F\cup G\right\vert }\right)
=\left( \frac{\pi \left( A\cap B\right) }{\pi \left( A\cup B\right) },\frac{%
\sigma \left( F\cap G\right) }{\sigma \left( F\cup G\right) }\right)
=sim(\alpha ,\beta )\left( F_{A},G_{B}\right) .$

\section{A representation method based on cardinality of soft hybrid spaces}

\begin{definition}
Let $F_{A}$ be a soft hybrid sets over $U$, i.e., $F_{A},G_{B}\in X(U)$.
Then depth of $A$ denoted by $depth\left( F_{A}\right) $ is given by 
\begin{equation*}
depth\left( F_{A}\right) =\left( m,mn\right) -\left( a_{1},a_{2}\right)
\end{equation*}%
where $a_{1}=\tsum\nolimits_{i=1}^{m}\mu _{A}(x_{i})$ and $%
a_{2}=\tsum\nolimits_{i=1}^{m}\tsum\nolimits_{j=1}^{n}f_{A}(e_{i})(x_{j})$
for $i=1,2,..,m$ and $j=1,2,..,n.$
\end{definition}

It is cleat that $depth\left( \tilde{U}\right) =0$ and $depth\left( \Phi
\right) =\left( m,mn\right) .$

\begin{definition}
Let $F_{A}$ and $G_{B}$ be two soft hybrid sets over $U$ , i.e., $%
F_{A},G_{B}\in X(U)$. Then we say $G_{B}$ is a better representative of $%
\tilde{U}$ than $F_{A}$ denoted by $G_{B}\supset $ $F_{A}$, if and only if 
\begin{equation*}
\left\Vert depth\left( G_{B}\right) \right\Vert <\left\Vert depth\left(
F_{A}\right) \right\Vert
\end{equation*}%
where $\left\Vert \left( a,b\right) \right\Vert $ is given by $\frac{%
|a|+\left\vert b\right\vert }{2}$.
\end{definition}

\begin{example}
Let $F_{A},G_{B},H_{C},K_{D}\in \mathcal{FPFS(}U\mathcal{)}.$ Suppose that $%
U=\left\{ x_{1},x_{2},x_{3},x_{4},x_{5}\right\} $ be a universal set and $%
E=\left\{ e_{1},e_{2},e_{3},e_{4}\right\} $ be a set of parameters, $%
A=\left\{ 0.1/e_{2},0.4/e_{3},0.2/e_{4}\right\} $, $B=\left\{
0.2/e_{2},0.5/e_{3},0.4/e_{4}\right\} ,C=\left\{
0.3/e_{2},0.3/e_{3},0.1/e_{4}\right\} $ and $D=\left\{
0.4/e_{2},0.1/e_{3},0.2/e_{4}\right\} .$ We consider the sets given as
follows:%
\begin{eqnarray*}
F_{A} &=&\left\{ 
\begin{array}{c}
\left\langle 0.1/e_{2},\left\{ 0.5/x_{1},0.1/x_{3},0.7/x_{4}\right\}
\right\rangle \\ 
\left\langle 0.4/e_{3},\left\{ 0.2/x_{3},0.4/x_{4},0.3/x_{5}\right\}
\right\rangle \\ 
\left\langle 0.2/e_{4},\left\{ 0.5/x_{2},0.1/x_{3},0.7/x_{4}\right\}
\right\rangle%
\end{array}%
\right\} ,G_{B}=\left\{ 
\begin{array}{l}
\left\langle 0.2/e_{2},\left\{ 0.3/x_{1},0.2/x_{3},0.5/x_{4}\right\}
\right\rangle \\ 
\left\langle 0.5/e_{3},\left\{ 0.2/x_{3},0.4/x_{4},0.3/x_{5}\right\}
\right\rangle \\ 
\left\langle 0.4/e_{4},\left\{ 0.5/x_{2},0.1/x_{3},0.7/x_{4}\right\}
\right\rangle%
\end{array}%
\right\} \\
H_{C} &=&\left\{ 
\begin{array}{c}
\left\langle 0.3/e_{2},\left\{ 0.4/x_{1},0.3/x_{3},0.1/x_{4}\right\}
\right\rangle \\ 
\left\langle 0.3/e_{3},\left\{ 0.1/x_{3},0.1/x_{4},0.3/x_{5}\right\}
\right\rangle \\ 
\left\langle 0.1/e_{4},\left\{ 0.6/x_{2},0.5/x_{3},0.4/x_{4}\right\}
\right\rangle%
\end{array}%
\right\} ,K_{D}=\left\{ 
\begin{array}{l}
\left\langle 0.4/e_{2},\left\{ 0.3/x_{1},0.4/x_{3},0.4/x_{4}\right\}
\right\rangle \\ 
\left\langle 0.1/e_{3},\left\{ 0.4/x_{3},0.2/x_{4},0.5/x_{5}\right\}
\right\rangle \\ 
\left\langle 0.2/e_{4},\left\{ 0.1/x_{2},0.2/x_{3},0.6/x_{4}\right\}
\right\rangle%
\end{array}%
\right\}
\end{eqnarray*}%
Then $\left\Vert depth\left( F_{A}\right) \right\Vert =\left\Vert \left(
4,20\right) -\left\{ \left( \tsum\nolimits_{i=1}^{m}\mu
_{A}(x_{i}),\tsum\nolimits_{i=1}^{m}\tsum%
\nolimits_{j=1}^{n}f_{A}(e_{i})(x_{j})\right) =\left( 0.7,3.5\right)
\right\} \right\Vert $

$=\left\Vert \left( 4,20\right) -\left( 0.7,3.5\right) \right\Vert =\frac{%
3.3+16.5}{2}=9.90.$

Similarity,$\left\Vert depth\left( G_{B}\right) \right\Vert =9.85,$ $%
\left\Vert depth\left( H_{C}\right) \right\Vert =10.25$ and $\left\Vert
depth\left( K_{D}\right) \right\Vert =10.10$. So we have the ranking $%
G_{B}\supset F_{A}\supset K_{D}\supset H_{C}.$ Thus $G_{B}$ is the best
representative of $U.$
\end{example}

\section{Conclusion}

In this paper, we firstly defined the concept of cardinality of soft hybrid
sets. Then we discussed the entropy, similarity and subsethood measures
based on cardinality. The relationships among these concepts was
investigated as well as related examples. An application of cardinality is
presented as a method for representation of a soft hybrid spaces. We hope
that the findings in this paper will help the researchers to enhance and
promote the further study on this concepts to carry out general framework
for the applications in practical life.

\end{document}